\documentclass[reqno,11pt]{article}
\usepackage[utf8]{inputenc} 
\usepackage{hyperref}
\hypersetup{                    
    colorlinks=true,                
    breaklinks=true,                
    urlcolor= green,                 
    linkcolor= blue,                
    citecolor= blue                
}
\usepackage{amsfonts,amsmath,amsthm, url}
\usepackage{amssymb}
\usepackage{graphics}
\usepackage{graphicx}
\usepackage{epsfig,color}
\usepackage{caption} 

\newtheorem{thm}{Theorem}
\newtheorem{claim}{Claim}
\newtheorem{as}{Assumption}
\newtheorem{prop}{Proposition}
\newtheorem{lem}{Lemma}[section]

\newtheorem{rem}{Remark}[section]

\newcommand{\esp}[1]{\mathbb{E}\left[ #1 \right]}
\newcommand{\espc}[2]{\mathbb{E}\left[ #1 \vert #2 \right]}
\newcommand{\prob}[1]{\mathbb{P}\left( #1 \right)}

\newcommand{\F}{\mathcal{F}}

\newcommand{\normL}[2]{\left|#2\right|_{L^{#1}}}

\newcommand{\normSigma}[2]{\left|#2\right|_{\Sigma^{#1}}}
\newcommand{\norm}[2]{\left| #2 \right|_{ #1 }}
\newcommand{\N}{\mathbb{N}}
\newcommand{\R}{\mathbb{R}}

\newcommand{\C}{\mathbb{C}}

\newcommand{\abs}[1]{\left|#1\right|}

\newcommand{\supp}{\operatornamewithlimits{Supp}}

\renewcommand{\epsilon}{\varepsilon}

\usepackage{dsfont}
\newcommand{\indic}[1]{\mathds{1}_{\left\{#1\right\}}}

\usepackage{upgreek}
\usepackage{textgreek}

\makeatletter
  
  \@addtoreset{equation}{section}
\makeatother
\begin{document}
\title{\textbf{Numerical analysis of the Gross-Pitaevskii Equation with a
randomly varying potential in time}}
\date{}
\author{Romain Poncet$^{\scriptstyle 1}$}



\maketitle

\begin{center}
$^1$ \textit{CMAP, Ecole Polytechnique, CNRS, Université Paris-Saclay, 91128 Palaiseau, France};\\
romain.poncet@cmap.polytechnique.fr
\end{center}

\vskip 0.1 in
\noindent
{\small
{\bf Abstract}.
The Gross-Pitaevskii equation with white noise in time perturbations of the
harmonic potential is considered.
In this article we define a Crank-Nicolson scheme based on a spectral discretization
and we show the convergence of this scheme in the case of cubic non-linearity and
when the exact solution is uniquely defined and global in time.
We prove that the strong order of convergence in probability is at least one.
\vskip 0.1 in
\noindent
\textit{keywords}: Gross-Pitaevskii equation; numerical analysis; stochastic partial differential
equations, nonlinear Schr\"{o}dinger equation


\section{Introduction}
The Gross-Pitaevskii equation (GPe) with cubic nonlinearity and quadratic potential is
used to model the evolution of Bose-Einstein Condensate (BEC) macroscopic wave
function in an all-optical far-off-resonance laser trap.
Fluctuations of laser intensity lead to consider white noise in time
perturbations of the harmonic potential (\cite{abdu,PhysRevA.69.053607}).
More precisely, we will be interested in the following dimensionless equation:
\begin{equation}
  \label{eq:GP}
  \left\{
    \begin{aligned}
      &id\phi -\lambda\abs{\phi}^{2\sigma}\phi dt - (-\Delta\phi+|x|^2\phi)dt=|x|^2\phi\circ dW_t,\\
      &\phi(0)=\phi_0,
    \end{aligned}
  \right.
\end{equation}
where the unknown $\phi$ is a random field taking complex values, on a probability space $(\Omega,\F,\mathcal{P})$,
depending on $t\geq 0$ and $x\in\R^d$.
We take $\sigma>0$ and $\lambda=\pm 1$.
Here, $(W_t)_{t\geq 0}$ is a Brownian motion taking real values associated with the filtration $(\F_t)_{t\geq 0}$.
$\circ$ denotes a Stratonovich product.
Although the stochastic integral associated with this product does not verify the martingale property as It\^{o}'s product does,
this choice of model is natural since the noise is actually a \emph{physical noise}
which in the real physical case is not exactly white, but has small correlation length.
A good way to understand this is to refer to \cite{dbf12} where the authors explain that equation \eqref{eq:GP}
can be seen in the subcritical case as a limit of equations where the Stratonovich product is replaced by
a random process with nonzero correlation length.
\vspace{3mm}

Existence and uniqueness of a solution for \eqref{eq:GP} has been proven in \cite{deBouardFukuizumi2007}
by use of a compactness method in the case $d=1,2$ with restrictions on $\sigma$.
The proof has been then generalized in \cite{dbf12} to the case $d\geq 3$ thanks to a dispersive estimate
for the linear equation.
\vspace{3mm}

The use of a Crank-Nicolson scheme is natural since mid-point discretization of the noise is consistent with the
Stratonovich product.
For example such a discretization has been used in
\cite{bdbd13,dbd04,Gazeau14}.
Moreover, as this discretization preserves the $L^2$-norm, it has also been used to discretize the deterministic equation
(see \cite{AntoineBaoBesse2013,BaoJakschMarkowich2003}) and nonlinear Schr\"{o}dinger equations with additive noise
(see \cite{dbd04}).
\vspace{3mm}

One of the main difficulties of this work is to prove stability for the naive Crank-Nicolson scheme in the linear case.
Lack of stability is also encountered in \cite{dbd04}, but actually comes from the
non-linearity without truncation.
Thus our case-study is more complicated.
To estimate the growth of the solution between two time-steps, one can express the norm (in which we expect stability)
of the solution at a given time-step with respect to the solution at the previous time-step.
Such an equation is the discrete analogous to the It\^{o} formula in the continuous case.
However, this method yields some new terms that we do not know how to estimate,
as explained in Remark \ref{rem:stab_lin}.
The solution we chose is to replace the unbounded operator $\abs{x}^2$
by a bounded approximation such that we can control its operator norm
with respect to $\delta t^{-1/4}$.
This approximation is built with the Hermite functions.
\vspace{3mm}

Another difficulty comes from the non-Lipschitz behavior of the nonlinearity.
We cannot ensure existence of a global solution for the discrete scheme by
the classical use of a fixed point method.
To circumvent this difficulty we use a classical argument
(already used in \cite{Antoine20132621,bdbd13,dbd04})
that consists in truncating the nonlinearity to make it Lipschitz,
and then making this truncation go to infinity.
This argument leads to a weaker sense of convergence for the untruncated equation
compared to the truncated equation.
Indeed, we are only able to get an order of convergence in probability instead of
mean-square.
This kind of result is classical for globally non-Lipschitz nonlinearities
\cite{bdbd13,Chen2015,LiuSplit}.
\vspace{3mm}

Besides, split-step methods have been used to approximate SPDEs and SDEs.
One of the main benefit is that they offer an easy and consistent way to approximate
the Stratonovich integral.
In \cite{LiuSplit} the author proposes a splitting scheme to solve the deterministic nonlinear
Schr\"{o}dinger equation and the stochastic nonlinear Schr\"{o}dinger equation with multiplicative
noise of Stratonovich type.
In both cases the nonlinearity is chosen to be non globally Lipschitz.
For the deterministic scheme an order of convergence is given for small enough
integration time interval, and in the stochastic case an order of convergence in
probability is given on a random time interval.
The proof uses a truncation on the nonlinearity.
Moreover, in \cite{DuboscqMartySplit}, the authors propose a Lie time-splitting
scheme for a nonlinear partial differential equation driven by a random
time-dependent dispersion coefficient.
In this case the nonlinearity is supposed to be Lipschitz, but the dispersion
coefficient can approximate a fractional Brownian motion.

We introduce some notations. We define for $k\in\N$ and $p\in\N^*$,
\begin{align}
  \label{eq:Sigmakp}
  \Sigma^{k,p}(\R^d)=\left\{ v\in L^p(\R^d),\displaystyle\sum_{\abs{\alpha}+\abs{\beta}\leq k}\normL{p}{x^\beta\partial^\alpha v}^p=\norm{\Sigma^{k,p}}{v}^p<+\infty \right\}.
\end{align}
We simply denote $\Sigma^{k}$ instead of $\Sigma^{k,2}$, and $\Sigma$ instead of $\Sigma^{1}$.
We denote by $C(\cdot)$ the constants, specifying dependencies with $\cdot$.
These constants can vary from one line to another.
We denote by $\langle \cdot,\cdot \rangle$ the usual Hermitian inner product in $L^2$,
defined for all $u,v\in L^2(\R^d)$ by $\langle u,v \rangle=\int_{\R^d} u(x)\overline{v}(x)dx$.
We denote by $A$ the operator $-\Delta +\abs{x}^2$.
In dimension $d=1$, we recall that this operator has purely discrete eigenvalues $\lambda_k=2k+1$,
with $k\in\N$.
The corresponding eigenfunctions are the Hermite functions, denoted by
$\{e_k,\quad k\in\N\}$, that form a complete orthonormal system in $L^2(\R)$.
In higher dimensions, the tensorization of this basis also form a complete orthonormal
system in $L^2(\R^d)$.
For example with $d=2$, we get the basis $\{e_{k,j}:=(x,y)\mapsto e_k(x)e_j(y),\quad k,j\in\N\}$,
and for $k,j\in\N$, the eigenvalue of $e_{k,j}$ is $2(k+j)+d$.
Moreover the sesquilinear form defined on $\Sigma^k(\R^d)\times\Sigma^k(\R^d)$ by
$(u,v)\mapsto\langle A^k u,v\rangle$ defines a Hermitian inner product on $\Sigma^k$,
and the associated norm is equivalent with the the norm defined in \eqref{eq:Sigmakp}.
For $K\in\N$, we denote by $\Sigma_K(\R)$ the finite dimensional
vector space spanned by the $K$ first Hermite functions.
In the multidimensional case we denote by $\Sigma_K(\R^d)$ the
finite dimensional subspace of $\Sigma(\R^d)$, of dimension $K^d$,
obtained by tensorisation of the space $\Sigma_K(\R)$ (see \cite{reed1975fourier}).
We denote by $P_K$ the orthogonal projection of $L^2(\R^d)$ onto $\Sigma_K(\R^d)$,
and we set $A_K=A P_K$.
\vspace{3mm}

We give now some notations about the the numerical scheme.
Let $T>0$ be the time horizon, and $N\in\N^*$ the number of time discretization steps.
We denote by $\delta t=T/N$ this step.
We also set $t_n:=n\delta t$ for $n\leq N$.
We set for all $n\in\N$, $\chi^{n+1}=\delta t^{-1/2}(W_{(n+1)\delta t}-W_{n\delta t})$.
Thus $(\chi^n)_{n\in\N^*}$ is a sequence of independent standard normal deviates.
We set $(\F_n)_{n\in\N}$ the filtration defined for all $n\in\N$ by $\F_n=\sigma(W_s,s\leq t_n)$,
so that for all $n\in\N^*$ $\chi^n$ is $\F_n$-measurable.
\vspace{3mm}

In section \ref{sec:def_scheme} we define the numerical scheme.
We introduce a truncation of the Brownian motion increments (that will enable us
to control the growth of the solution of the scheme), and of the non-linearity.
We also introduce an approximation of the term $\abs{x}^2$, that will enable us
to prove stability.
In section \ref{sec:well-posedness} we prove existence and uniqueness of a
solution for the numerical scheme. We also prove stability.
The stability relies heavily on the spectral discretization.
We prove then in section \ref{sec:conv} the convergence in probability of our scheme toward the
GP equation, with a strong order at least one.
We present in section \ref{sec:num} numerical simulations for the numerical
scheme, and some convergence comparisons with various other schemes.
\vspace{3mm}

\section{Definition of the numerical scheme and main result}
\label{sec:def_scheme}

This section is devoted to the construction of the numerical scheme and to the
statement of the main convergence result.
The technical precisions are quite cumbersome and interfere with the comprehension
of the general idea.
For the sake of clarity we first present a refined version of the main statements in section
\ref{subsec:objectives}, introducing the main ideas and the technical issues.
We then give a rigorous definition of the scheme that enables us to state the convergence result.

\subsection{Objectives and formal result}
\label{subsec:objectives}

As stated in introduction, equation \eqref{eq:GP} is known to be well-posed
under some assumptions on $\lambda$, $\sigma$ and $d$ (see \cite{deBouardFukuizumi2007,dbf12,Bouard20152793}).
We state a slightly stronger result in the specific case $\lambda=1$, $\sigma=1$ and
$d\leq 3$ for arbitrarily regular initial condition.
All the results of this article hold only in this specific case since we require the following result,
\begin{prop}
  \label{prop:exist_Sk}
  Assume $\lambda=1$, $\sigma=1$ and $d\leq 3$. Then $\forall j\in\N^*$,
  if $\phi_0\in\Sigma^j(\R^d)$ then there exists a unique global solution $\phi$
  of \eqref{eq:GP} adapted to $(\F_t)_{t\geq 0}$, with $\phi(0)=\phi_0$,
  almost surely in $C(\R^+;\Sigma^j(\R^d))$.
\end{prop}
The case $j=1$ and $\sigma<2$ if $d=3$ or $\sigma<+\infty$ if $d=2$, has been
initially partially treated in \cite{deBouardFukuizumi2007} and generalized in \cite{dbf12}.
Then, the case $j=2$, $d=2$ and $\sigma\geq 1/2$ has been treated in \cite{Bouard20152793}.
The proof is given in Appendix and is done recursively using the idea introduced
in \cite{Bouard20152793}.
This extra regularity is needed in order to show an order of convergence of the numerical scheme
(that will be introduced thereafter) toward the solution of equation \eqref{eq:GP}.
\vspace{3mm}

The naive Crank-Nicolson discretization of Equation \eqref{eq:GP} would be given by,
\begin{align}
  \label{eq:scheme_refine}
  \phi^{n+1}-\phi^n=
  -i\left(\delta t A +\sqrt{\delta t}\chi^{n+1}\abs{x}^2\right)\left(\dfrac{\phi^{n+1}+\phi^n}{2}\right)
  -i\lambda\delta t g(\phi^{n+1},\phi^n),
\end{align}
where $g$ is a classical approximation of the non-linearity \cite{Delfour1981277}
given by
\begin{align}
  \label{eq:g}
  g(\phi^{n+1},\phi^n)=
  \dfrac{1}{\sigma +1}
  \left(\dfrac{\abs{\phi^{n+1}}^{2\sigma+2}-\abs{\phi^{n}}^{2\sigma+2}}{\abs{\phi^{n+1}}^{2}-\abs{\phi^{n}}^{2}}\right)
  \left(\dfrac{\phi^{n+1}+\phi^n}{2}\right),
\end{align}
which can be simplified (to get rid of the singularity) in our case $\sigma=1$ by
\begin{align*}
  g(\phi^{n+1},\phi^n)=
  \dfrac{1}{2}\left(\abs{\phi^n}+\abs{\phi^{n+1}} \right)\left(\dfrac{\phi^{n+1}+\phi^n}{2}\right).
\end{align*}

It can be formally written in the following form
\begin{align}
  \label{eq:scheme_refine2}
  \phi^{n+1}=S_{\delta t,n+1}\phi^n-i\lambda\delta t T^{-1}_{\delta t,n+1}g(\phi^{n+1},\phi^n),
\end{align}
where
\begin{align*}
  T_{\delta t,n+1}&=\left(Id+i\dfrac{\delta t}{2}A+i\dfrac{\sqrt{\delta t}}{2}\chi^{n+1}\abs{x}^2\right),\\
  S_{\delta t,n+1}&=\left(Id+i\dfrac{\delta t}{2}A+i\dfrac{\sqrt{\delta t}}{2}\chi^{n+1}\abs{x}^2\right)^{-1}
  \left(Id-i\dfrac{\delta t}{2}A-i\dfrac{\sqrt{\delta t}}{2}\chi^{n+1}\abs{x}^2\right).
\end{align*}
\vspace{3mm}

Let us suppose that this scheme is well-defined in the following sense,
\begin{claim}
  \label{claim1}
  There exists a unique discrete solution $\mathcal{F}_n$-adapted $\phi=(\phi^n)_{n=0,\ldots,N}$
  satisfying \eqref{eq:scheme_refine} which belongs a.s. to $L^\infty(0,T;\Sigma^j)$.
\end{claim}
Then we want to establish a convergence result in probability for this discrete solution
towards the solution of equation \eqref{eq:GP}:
\begin{claim}
  \label{claim2}
  For all $k\in\N^*$, $\phi_0\in\Sigma^{k+12}$ and $\alpha<1$,
  there exists $C>0$ such that
    \begin{align*}
    \displaystyle\lim_{\delta t\rightarrow 0}\sup_{n\delta t\leq T}\prob{\normSigma{k}{\phi^n-\phi(t_n)}> C\delta t^\alpha}=0,
  \end{align*}
  where $(\phi(t))_{t\in[0,T]}$ denotes the solution of \eqref{eq:GP} given by
  Proposition \ref{prop:exist_Sk}.
\end{claim}
The respective rigorous analogous of these claims are Proposition \ref{prop:exist_scheme}
and Theorem \ref{thm:conv_proba}.
If we aim to show a convergence in probability, and not a stronger convergence,
as mean-square for example, it is because of the non-Lipschitz nonlinearity.
To show this result, we actually prove an intermediary result of mean-square convergence
when the nonlinearity is replaced by a Lipschitz truncation.
We lose this strong convergence by making the level of truncation go to infinity.
Similar methods have been used in \cite{bdbd13,Chen2015,dbd04}.
\vspace{3mm}

We did not manage to prove these two claims for the scheme given by \eqref{eq:scheme_refine2}.
The main issues are linked to the invertibility of the operator $T_{\delta t,n+1}$,
and to the stability of the linear equation.
The goal of the section \ref{subsec:rig_def} is to slightly modify operators
$T_{\delta t,n+1}$ and $S_{\delta t,n+1}$ so that we are able to prove
Claims \ref{claim1} and \ref{claim2}.
More precisely these modifications should take into account that for strictly
positive values of $j$, invertibility of operators $T_{\delta t,n+1}$
when $\chi^{n+1}\delta t$ is very negative is not clear.
Thereafter we propose to circumvent this difficulty by truncating
the increments of the Brownian motion.
Nevertheless this modification is not sufficient to enable us to prove Claim \ref{claim2}.
In order to be able to prove it, we introduce thereafter an approximating sequence
$(B_K)_{K\in\N^*}$ of the operator $\abs{x}^2$, composed by bounded operators
in every $\Sigma^l$ for $l\in\N$.
By replacing $\abs{x}^2$ by $B_K$, we are then able to prove the stability.
The next subsection defines rigorously such modifications on operators $T_{\delta t,n+1}$
and $S_{\delta t,n+1}$ that enable us to show a convergence in probability toward
Equation \eqref{eq:GP}.

\subsection{Rigorous definition of the scheme}
\label{subsec:rig_def}

In all the following $k\in\N$ refers to the space $\Sigma^k(\R^d)$ in which
convergence of the numerical scheme is expected, and $j\in\N$, refers to the space
$\Sigma^j(\R^d)$ in which stability is expected.
To be able to prove an order of convergence, we require $j$ to be greater than $k$.
More precisely, our convergence theorem will require $k>d/2$ and $j=k+12$.
\vspace{3mm}

We first introduce some objects.
For all $K\in\N^*$, we set $B_K$, a self-adjoint operator in $L^2$ that we will specify later.
It is essentially an approximation of the operator $\abs{x}^2$ which is a bounded
operator in every $\Sigma^j$ for $j\in\N$.
We furthermore suppose that for all $K\in\N^*$, $B_K$ takes its values into $\Sigma_K$,
for making the scheme implementable.

We set now $\theta$ and $\theta_j$, regular truncations of the Heaviside function
such that $\theta\in\mathcal{C}^\infty(\R^+),\theta\geq 0$, $\supp{\theta}\subset[0,2]$
and $\theta\equiv 1$ on $[0,1]$.
And we set for all $L\in\N^*$, for all $x\in\R^+$, $\theta_L(x)=\theta(x/L)$.
We introduce $f_L$ and $g_L$, two Lipschitz approximations of the non-linearity defined by,
\begin{align}
  \label{eq:f_L}
  &\forall \phi\in \Sigma^k,\quad g^k_L(\phi_1,\phi_2)=
  \theta_L(\normSigma{k}{\phi_1}^2)
  \theta_L(\normSigma{k}{\phi_2}^2)
  \dfrac{1}{2}\left(\abs{\phi_1}^2+\abs{\phi_2}^2\right)
  \left(\dfrac{\phi_1+\phi_2}{2}\right),\\
  \label{eq:f_L2}
  &\forall \phi\in \Sigma^k,\quad f^k_L(\phi)=g^k_L(\phi,\phi).
\end{align}

For all $C_0>0$, we introduce the $\F_n$-measurable stopping time $\tau_{\delta t,C_0}$ defined by,
\begin{equation}
  \label{eq:tau_lin}
   \tau(\delta t,C_0)=\min\left(
   \left\{n\in\N,1\leq n\leq N,\abs{W_{t_n}-W_{t_{n-1}}} \geq C_0\right\}\cup\left\{N+1\right\}
   \right).
\end{equation}
We introduce this stopping time for technical purposes.
It prevents the most shaken up trajectories of the Brownian motion
to downgrade the regularity of the discrete solution.
In the regime $\delta t$ small, $\tau(\delta t,C_0)=N+1$ with high probability,
as stated by the following Lemma.
It shows that the case $\tau(\delta t,C_0)<N+1$ happens with small
probability when $\delta t$ vanishes and thus corresponds to
pathological cases.
\begin{lem}
  \label{lem:tau_conv}
  For all $C_0>0$,
  \begin{align*}
    \prob{\tau(\delta t,C_0)<N(\delta t)+1} = O(N e^{-\frac{C_0^2 N}{2T^2}}).
  \end{align*}
\end{lem}
\begin{proof}[Proof of Lemma \ref{lem:tau_conv}]
  The domination comes from the fact that
  \begin{align*}
    \prob{\tau(\delta t,C_0)<N(\delta t)+1}
    &\leq \sum_{1\leq n\leq N} \prob{\abs{W_{t_n}-W_{t_{n-1}}}\geq C_0}\\
    &\leq N \prob{\abs{G}\geq \dfrac{C_0\sqrt{N}}{\sqrt{T}}},
  \end{align*}
  where $G$ denotes a standard normal deviate.
  The result comes from the fact that for all $x>0$, $\prob{\abs{G}\geq x}\leq 2 e^{-x^2/2}$.
\end{proof}

We define the random operators $T^{-1}_{\delta t,K,n+1,C_0}$ and $S_{\delta t,K,n+1,C_0}$
on $\Sigma_K(\R^d)$ for $n\leq N$ by,
\begin{equation}
  \label{eq:T}
  T^{-1}_{\delta t,K,n+1,C_0}=
  \left\{
  \begin{aligned}
    &\left(Id+i\frac{\delta t}{2} A_K+i\frac{\sqrt{\delta t}}{2}\chi^{n+1}B_K\right)^{-1},
    \quad\text{if}\quad n+1<\tau_{C_0,\delta t}\\
    &0,
    \quad\text{otherwise}
  \end{aligned}
  \right.
\end{equation}
\begin{equation}
  \label{eq:S}
  S_{\delta t,K,n+1,C_0}=
  \left\{
  \begin{aligned}
    &T_{\delta t,K,n+1,C_0}^{-1}
    \left(Id-i\frac{\delta t}{2} A_K-i\frac{\sqrt{\delta t}}{2}\chi^{n+1}B_K\right),
    \quad\text{if}\quad n+1<\tau_{C_0,\delta t},\\
    &Id,
    \quad\text{otherwise}.
  \end{aligned}
  \right.
\end{equation}
Invertibility of the operator $\left(Id+i\frac{\delta t}{2} A_K+i\frac{\sqrt{\delta t}}{2}\chi^{n+1}B_K\right)$
follows from the fact that it will be chosen to be symmetric for the Hermitian
inner product in $L^2$ (see Assumption \ref{as:1}),
and that $\Sigma_K(\R^d)$ is finite-dimensional.

We present now our Crank-Nicolson scheme, with a spectral discretisation in space.
More precisely our scheme is defined on the space spanned by the tensorisation
of the $K$ first Hermite functions.
We set the family of stochastic processes $(\phi^n_{K,L,\delta t})$,
indexed by $K\in\N^*$, $L\in\N^*$ and all possible values of $\delta t$, defined by,
\begin{equation}
  \label{eq:scheme}
  \left\{
  \begin{aligned}
    \phi^{n+1}_{K,L,\delta t}&=S_{\delta t,K,n+1,C_0}\phi^n_{K,L,\delta t}-i\delta t\lambda T_{\delta t,K,n+1,C_0}^{-1}P_K g^k_L(\phi^n_{K,L,\delta t},\phi^{n+1}_{K,L,\delta t}),\\
    \phi^0_{K,L,\delta t}&=P_K \phi_0.
  \end{aligned}
  \right.
\end{equation}
Existence and uniqueness of a solution for \eqref{eq:scheme} is not obvious since this
equation is non-linearly implicit.
We show in the next section that it is actually well-posed.
The approximation of the non-linearity is classical, see \cite{bdbd13,dbd04,Delfour1981277},
and comes from the fact that it is conservative for the energy in the deterministic case
(without truncation).
\vspace{3mm}

We specify now how to choose the family of operators $(B_K)_{K\in\N^*}$.
We require it to satisfy the following assumptions where we denote
by $|||\cdot|||_{\mathcal{L}(\Sigma^m,\Sigma^n)}$ the operator norm for
the linear applications from $\Sigma^m$ to $\Sigma^n$.
\begin{as}
  \label{as:1}
   $(B_K)_{K\in\N^*}$ is such that for all $k\in\N$, there exists $C(k)>0$ such that,
  \begin{align}
    &B_K=B_K^*\label{as:eq:0},\quad\forall K\in\N^*,\\
    &|||B_K|||_{\mathcal{L}(\Sigma^{k+2},\Sigma^k)}\leq C(k)\label{as:eq:1},\quad\forall K\in\N^*,\\
    &|||B_K|||_{\mathcal{L}(\Sigma^k,\Sigma^k)}\leq C(k)K\label{as:eq:2},\quad\forall K\in\N^*,\\
    &|||[A^k,B_K]|||_{\mathcal{L}(\Sigma^{2k},L^2)} \leq C(k)\label{as:eq:3},\quad\forall K\in\N^*,\\
    &\forall\phi\in\Sigma^k,\quad \abs{\Re\langle [A^k,B_K]\phi,B_K\phi \rangle} \leq C(k)\normSigma{k}{\phi}^2\label{as:eq:4},\quad\forall K\in\N^*.
  \end{align}
\end{as}
\begin{as}
  \label{as:convergence}
  We suppose that there exists $C>0$ such that for all $K\in\N$, for all $k\in\N$,
  for all $p\in\N$,
  and for all $\phi\in\Sigma^{k+p+2}(\R^d)$,
  \begin{align}
    \label{as:eq:6}
    \normSigma{k}{(\abs{x}^2-B_K)\phi}^2\leq C K^{-p}\normSigma{k+p+2}{\phi}^2.
  \end{align}
\end{as}
Assumption \ref{as:1} is required for proving the stability of the numerical scheme.
Assumption \ref{as:convergence} specifies the convergence of the sequence $B_K$ toward
the operator $\abs{x}^2$.
Properties \eqref{as:eq:0}, \eqref{as:eq:1}, \eqref{as:eq:3} and \eqref{as:eq:4}
are satisfied by the operator $\abs{x}^2$, and are natural to impose to the sequence $(B_K)$,
for technical purposes.
The essential difference between $(B_K)$ and $\abs{x}^2$ stated in \eqref{as:eq:2}
is that the $B_K$ are not unbounded operators from $\Sigma^k$ to $\Sigma^k$,
contrary to $\abs{x}^2$.
This is why we expect $|||B_K|||_{\mathcal{L}(\Sigma^k,\Sigma^k)}$ to tend to $+\infty$
by requiring $B_K$ to tend to $\abs{x}^2$ when $K$ increases.
Moreover \eqref{as:eq:2} actually controls the speed at which this operator norm diverges.
This speed comes from the fact that there exists $C>0$ such that for all $k,p\in\N$
and for all $\phi\in\Sigma^{k+p}(\R^d)$ and for $K\in\N^*$,
\begin{align}
  \label{eq:estim_trunc}
  \normSigma{k+p}{P_K \phi}^2\leq C K^p \normSigma{k}{\phi}^2.
\end{align}
Assumption \ref{as:convergence} is natural since there exists $C>0$, such that
for all $k,p\in\N$, for all $K\in\N$ and for all $\phi\in\Sigma^{k+p}(\R^d)$,
\begin{align}
  \label{eq:diff_trunc}
  \normSigma{k}{(Id-P_K)\phi}^2\leq C K^{-p}\normSigma{k+p}{\phi}^2,
\end{align}
and for all $\phi\in\Sigma^{k+p+2}(\R^d)$,
\begin{align}
  \label{eq:diff_truncx2}
  \normSigma{k}{(Id-P_K)\abs{x}^2\phi}^2\leq C K^{-p}\normSigma{k+p+2}{\phi}^2.
\end{align}
Equations \eqref{eq:diff_trunc} and \eqref{eq:diff_truncx2} follows from
a direct computation based on the expansion of $\phi$ in eigenvectors of $A$.
\vspace{3mm}

To satisfy hypotheses \eqref{as:eq:2} and \eqref{as:eq:6}, one could think of choosing
$B_K=P_K\abs{x}^2P_K$, but then \eqref{as:eq:4} would not hold anymore.
To satisfy all of them, we can define $B_K$ as a smooth (actually Lipschitz)
spectral cutoff of $\abs{x}^2$, instead of a cutoff such as
$P_K \abs{x}^2P_K$.
We give now an example of such a family of operators.
We begin with constructing it in dimension one, and we then generalize to higher dimensions.
We fix $\theta\in]0,1[$, the parameter that specifies the smoothness of the cutoff,
and we set,
\begin{align}
  \label{eq:def_BK_1}
  \forall m\in\N,\quad B_K e_m=\alpha_m^K e_{m-2}+\beta_m^K e_m+\gamma_m^K e_{m+2},
\end{align}
with
\begin{equation}
  \label{eq:def_BK_2}
  \forall m\in\N,\forall K\in\N^*,\quad \alpha_{m}^K=
  \left\{
  \begin{aligned}
    &\dfrac{1}{2}\sqrt{(m-1)m}=\langle\abs{x}^2,e_{m-2}\rangle,\quad\text{if}\quad 2\leq m\leq \theta K\\
    &\alpha^K_{\lfloor\theta K\rfloor}\left(1-\left(\dfrac{m-\lfloor\theta K\rfloor}{K-\lfloor\theta K\rfloor}\right)\right),\quad\text{if}\quad \theta K<m\leq K\\
    &0,\quad\text{otherwise}
\end{aligned}
  \right.
\end{equation}
\begin{equation}
  \label{eq:def_BK_3}
  \forall m\in\N,\forall K\in\N^*,\quad \beta_{m}^K=
  \left\{
  \begin{aligned}
    &\dfrac{1}{2}(2m+1)=\langle\abs{x}^2,e_{m}\rangle,\quad\text{if}\quad 0\leq m\leq \theta K\\
    &\beta^K_{\lfloor\theta K\rfloor}\left(1-\left(\dfrac{m-\lfloor\theta K\rfloor}{K-\lfloor\theta K\rfloor}\right)\right),\quad\text{if}\quad \theta K<m\leq K\\
    &0,\quad\text{otherwise}\quad K<m
  \end{aligned}
  \right.
\end{equation}
\begin{equation}
  \label{eq:def_BK_4}
  \forall m\in\N,\forall K\in\N^*,\quad \gamma_{m}^K=\alpha_{m+2}^K.
\end{equation}
We recall that in dimension one, the operator $\abs{x}^2$ verifies,
\begin{align*}
\forall m\in\N,\quad \abs{x}^2 e_m=\dfrac{1}{2}\left( \sqrt{(m-1)m}e_{m-2}+(2m+1)e_m+\sqrt{(m+1)(m+2)}e_{m+2} \right),
\end{align*}
which explains why we want $(B_K)$ to satisfy \eqref{eq:def_BK_1}.
Moreover, \eqref{eq:def_BK_4} implies that $B_K$ is symmetric.
The extension to the $d$-dimensional case, with $d\geq 2$, is done by setting
for $i_1,i_2,\ldots,i_d \leq K$,
\begin{align}
  \label{eq:def_BK_5}
  B_K\left(\bigotimes_{k=1}^d e_{i_k}\right)
  =\sum_{j=1}^d \bigotimes_{k=1}^d \left(\indic{j=k}B_K+\indic{j\neq k}Id\right)e_{i_k}.
\end{align}
We refer to \cite{reed1975fourier} for basics about tensorisation.
For example, with $d=2$,
$B_K(e_{i_1}\otimes e_{i_2})=(B_K e_{i_1})\otimes e_{i_2}+e_{i_1}\otimes (B_K e_{i_2})$.
This extension is natural since it holds for operator $\abs{x}^2$.
\vspace{3mm}

In this paper, stability is proved by an extensive use of the boundedness of the operators $B_K$.
Convergence toward the solution of \eqref{eq:GP} will be achieved by making
$\delta t$ go to zero and $K$ and $L$ to infinity.
We will show that stability holds if $K$ does not tend to infinity quicker than $\delta t^{-1/4}$,
We will show then that convergence of order at least one in probability holds
if $K$ actually scales like $\delta t^{-1/4}$.
\vspace{3mm}

\section{Well-posedness of the numerical scheme}
\label{sec:well-posedness}

We begin with explaining how to choose $C_0$ in the definition \ref{eq:tau_lin}
of $\tau(\delta t,C_0)$.
We require in the following the operator $S_{\delta t,K,n+1,C_0}$ and
$T^{-1}_{\delta t,K,n+1,C_0}$ to be uniformly bounded in $K$.
It is possible to choose $C_0$ to satisfy this requirement as stated by the
following lemma,
\begin{lem}
  \label{lem:C_0}
  Under Assumption \ref{as:1},
  for all $j\in\N$, there exists $C_0(j)>0$ such that
  for all $K\in\N^*$, $\phi\in\Sigma_K(\R^d)$, $\delta t\in (0,1)$ and $n\in\N$,
  \begin{align*}
    \normSigma{j}{S_{\delta t,K,n+1,C_0(j)}\phi}^2 &\leq 2 \normSigma{j}{\phi}^2,\quad\text{a.s.}\\
    \normSigma{j}{T^{-1}_{\delta t,K,n+1,C_0(j)}\phi}^2 &\leq 2 \normSigma{j}{\phi}^2,\quad\text{a.s.}
  \end{align*}
\end{lem}
From now on, we suppose that $C_0(j)$ satisfies this lemma for a given $j$.
We denote then $S_{\delta t,K,n+1,j}$ and $T^{-1}_{\delta t,K,n+1,j}$.
This technical difficulty did not appear in \cite{dbd04} where the noise is multiplicative.
The difference comes from the fact that in our case the potential reduces the
regularity of the noise term.
One can note that this truncation on the noise would have been also necessary in
order to prove Claim \ref{claim1} for operators $T_{\delta t,n+1}$ and
$S_{\delta t,n+1}$ as stated before.
\begin{proof}[Proof of Lemma \ref{lem:C_0}]
  Let $\phi\in\Sigma_K(\R^d)$, $n\in\N$ and $C_0>0$. Then, to simplify notations, we set
  $\psi=S_{\delta t,K,n+1,C_0}\phi$.
  For almost every $\omega\in \{\tau(\delta t)\leq n+1\}$, $\psi=\phi$,
  where $\tau(\delta t)$ is defined by \eqref{eq:tau_lin},
  and there is nothing to prove.
  Then, we suppose that $\omega\in \{\tau(\delta t)> n+1\}$,
  which implies that $\sqrt{\delta t}\abs{\chi^{n+1}} < C_0$.
  The expression of $S_{\delta t,K,n+1,C_0(j)}$ enables to write,
  \begin{align}
    \label{eq:C_0_proof}
    \psi=\phi-i\left(A_K\delta t+B_K\sqrt{\delta t}\chi\right)\left(\dfrac{\phi+\psi}{2}\right).
  \end{align}
  Then, by taking the $L^2$ inner product of both sides of this equation with $A^j(\phi+\psi)$,
  one finds the following implicit expression of the growth of the solution :
  \begin{align*}
    \normSigma{j}{\psi}^2-\normSigma{j}{\phi}^2
    =&\Re\langle A^j(\phi+\psi),\psi-\phi \rangle\\
    =&-\dfrac{\sqrt{\delta t}}{2}\chi\Im\langle A^j(\phi+\psi),B_K(\phi+\psi)\rangle.
  \end{align*}
  Then, using first \eqref{as:eq:0}, and then \eqref{as:eq:3},
  \begin{align}
    \label{eq:trunc_maj}
    \normSigma{j}{\psi}^2-\normSigma{j}{\phi}^2
    &=\dfrac{\sqrt{\delta t}}{4}\chi\Im\langle [A^j,B_K](\phi+\psi),\phi+\psi\rangle\\
    &\leq \sqrt{\delta t}\abs{\chi} C(j)\left(\normSigma{j}{\psi}^2+\normSigma{j}{\phi}^2\right).
  \end{align}
  The last line exhibits the reason why we require the truncation on the Gaussian increment.
  Indeed, since $\omega\in \{\tau(\delta t)> n+1\}$, we get that
  \begin{align}
    \label{eq:trunc_maj2}
    \normSigma{j}{\psi}^2-\normSigma{j}{\phi}^2
    \leq C_0 C(j)\left(\normSigma{j}{\psi}^2+\normSigma{j}{\phi}^2\right),
  \end{align}
  and thus by choosing $C_0(j)=\frac{1}{3 C(j)}$, we get that
  \begin{align*}
    \normSigma{j}{\psi}^2\leq 2 \normSigma{j}{\phi}^2.
  \end{align*}
  The estimate on $T^{-1}_{\delta t,K,n+1,C_0(j)}$ is proven in a similar way.
\end{proof}
\vspace{3mm}

The next proposition states the well-posedness of the numerical scheme defined by
\eqref{eq:scheme}.
In this proposition we consider $K$ as a function of $\delta t$, so that $K$ may tend to infinity
when $\delta t$ tends to zero.
Since the non-linearity is implicit, well-posedness is not immediate.
It is classically shown by use of a Banach fixed point theorem,
which can be easily implemented.
Moreover, this proof relies on the Lipschitz truncation of the nonlinearity.
\begin{prop}
  \label{prop:exist_scheme}
  Under Assumption \ref{as:1},
  for all $k>d/2$,
  $j\geq k$,
  $L\in\N^*$,
  $T>0$
  and $K_0>0$,
  there exists $\delta t_0(j,L,K_0)>0$
  and $C(j,L,K_0,T)>0$ such that
  for all $\phi_0\in\Sigma^j(\R^d)$,
  for all $\delta t=T/N\leq\delta t_0(j,L,K_0)$,
  and $K\leq K_0\delta t^{-1/4}$,
  there exists a unique discrete solution $(\F_n)$-adapted
  $\phi=(\phi^n)_{n=0,\ldots,N}$
  satisfying \eqref{eq:scheme},
  almost surely in $L^2_\omega L^\infty([0,T];\Sigma^j(\R^d))$,
  and such that
  \begin{align*}
    \esp{\sup_{n\leq N}\normSigma{j}{\phi^n}^2}
    \leq
    C(j,L,K_0,T)\normSigma{j}{\phi_0}^2.
  \end{align*}
\end{prop}
It is important to notice that the norm in which we truncate the nonlinearity
is not linked to the space in which we prove existence and uniqueness,
but to the space in which we will prove the convergence of the numerical scheme.
The fact that $\phi$ belongs a.s. in $L^\infty(0,T;\Sigma^j(\R^d))$
is understood supposing $\phi$ constant on every interval
$[n\delta t,(n+1)\delta t[$ for $0\leq n\leq N-1$.
\vspace{3mm}

The most technical part in the proof of this proposition lies in showing
the stability in $\Sigma^j$ for the linear part of the scheme, given by
\begin{equation}
  \label{eq:scheme_lin}
  \left\{
  \begin{aligned}
    \phi^{n+1}&=S_{\delta t,K,n+1,C_0}\phi^n,\\
    \phi^0&=P_K \phi_0.
  \end{aligned}
  \right.
\end{equation}
Stability of the linear equation \eqref{eq:scheme_lin} is a main issue in this article.
We recall that we were not able to prove stability for the semi-discrete
Crank-Nicolson scheme defined by \eqref{eq:scheme_refine2}.
The proof we present uses extensively the boundedness of operators $B_K$,
and the truncation of the noise.
The next lemma states the stability result for \eqref{eq:scheme_lin}.
\begin{lem}
  \label{lem:stability_lin}
  For all $j\in\N$, $T>0$, $K\in\N^*$ and $N\in\N^*$ such that $\delta t=T/N$,
  there exists a constant $C(j,T,K^4\delta t)>0$ such that
  for all $\phi_0\in\Sigma^j(\R^d)$,
  \begin{equation}
    \label{eq:stability_lin}
    \esp{\sup_{n\leq N} \normSigma{j}{\phi^n}^2}
    \leq C(j,T,K^4\delta t)\normSigma{j}{\phi^0}^2,
  \end{equation}
  where $(\phi^n)_{n\leq N}$ denotes the solution of the linear scheme
  given by \eqref{eq:scheme_lin}.
\end{lem}
This Lemma enables to show stability for the linear equation on a fixed interval $[0,T]$,
scaling $K$ as $\delta t^{-1/4}$.
\begin{rem}
  \label{rem:stab_lin}
  The proof of Lemma \ref{lem:stability_lin} relies on
  the discrete analogous of the It\^{o} lemma for the continuous process $(\phi_t)_{t\geq 0}$
  solution of the linear part of \eqref{eq:GP},
  \begin{align*}
    &d\phi= -iA\phi dt -i|x|^2\phi\circ dW_t,\\
  \end{align*}
  Indeed, to show the stability of the numerical scheme, the idea is to compute
  the variation of the $\Sigma(\R^d)$ norm of the solution of the numerical scheme
  between two time steps.
  The computation is done in the proof of Lemma \ref{lem:stability_lin}
  (see Equation \eqref{eq:ito_discrete_linear}) and gives,
  \begin{align}
    \label{eq:discreteItobis}
    \normSigma{j}{\phi^{n+1}}^2=&\normSigma{j}{\phi^n}^2\\
    &+\sqrt{\delta t}\widetilde{\chi}^{n+1} \Im \langle [A^j,\abs{x}^2]\phi^n,\phi^n \rangle\\
    &+\dfrac{\sqrt{\delta t}}{2}\widetilde{\chi}^{n+1} \Im \langle [A^j,\abs{x}^2](\phi^{n+1}+\phi^n),(\phi^{n+1}-\phi^n)\rangle\\
    &-\dfrac{\sqrt{\delta t}}{4}\widetilde{\chi}^{n+1} \Im \langle [A^j,\abs{x}^2](\phi^{n+1}-\phi^n),(\phi^{n+1}-\phi^n)\rangle.
  \end{align}
  For the time continuous process, a similar computation using It\^{o}'s lemma gives,
  \begin{align}
    \label{eq:discreteIto}
    \normSigma{j}{\phi_{t_{n+1}}}^2=&\normSigma{j}{\phi_{t_{n}}}^2
    +\int_{t_n}^{t_{n+1}} \Im \langle [A^j,\abs{x}^2] \phi_s , \phi_s \rangle dWs
    +\int_{t_n}^{t_{n+1}} \Re \langle [A^j,\abs{x}^2] \phi_s ,\abs{x}^2 \phi_s \rangle ds.
  \end{align}
  Then, the first term in the right-hand side of Equation \eqref{eq:discreteItobis}
  is an approximation of the stochastic integral of Equation \eqref{eq:discreteIto},
  the second term in the right hand-side is an approximation of the It\^{o} correction
  in Equation \eqref{eq:discreteIto}.
  It is possible to estimate these two terms in the same way as it would be done
  to show at most exponential growth for \eqref{eq:discreteIto} (using Gronwall's inequality).
  Yet, the main problem is the apparition of the third term in the right-hand side
  of Equation \eqref{eq:discreteItobis}.
  We do not know how to estimate it.
  It is the reason why we make use of the spatial discretization
  and the approximation $B_K$ of the operator $\abs{x}^2$.
  \end{rem}

We prove now first Proposition \ref{prop:exist_scheme} and then
Lemma \ref{lem:stability_lin}.
\begin{proof}[\textbf{Proof of Proposition \ref{prop:exist_scheme}}]
  First, we recall that the approximation of the nonlinearity is Lipschitz
  as stated by the following lemma,
  \begin{lem}
    \label{lem:lips}
     Let $k>d/2$. Then for all $L\in\N^*$, there exists $C(L,k)>0$ such that,
    \begin{align*}
      \forall u_1,u_2,v_1,v_2\in\Sigma^k(\R^d),\quad\normSigma{k}{g^k_L(u_1,v_1)-g^k_L(u_2,v_2)}\leq C(L,k)(\normSigma{k}{u_1-u_2}+\normSigma{k}{v_1-v_2}).
    \end{align*}
    Moreover, for $d=1,2,3$, for all $j\geq k$, and for all $L\in\N^*$, there exists $C(j,L)>0$ such that,
    \begin{align*}
      \forall u,v\in\Sigma^j(\R^d),\quad \normSigma{j}{g^k_L(u,v)}\leq C(j,L)\left( \normSigma{j}{u}+\normSigma{j}{v} \right).
    \end{align*}
    where $g^k_L$ is defined by equation \eqref{eq:f_L}.
  \end{lem}
  The proof of Lemma \ref{lem:lips} is classical, and is omitted in this article.
  Let $\Delta<T$.
  We denote by $X(j,\Delta)$ the space of $(\F_n)$-measurable processes that belong to
  $L^2_\omega L^\infty(0,\Delta;\Sigma^{j}(\R^d))$.
  To prove well-posedness of Equation \eqref{eq:scheme} in $X(j,\Delta)$, we use a fixed-point
  method.
  It is clear that the solutions of \eqref{eq:scheme} are the fixed point of
  the application $\mathcal{T}$ defined for all $\phi\in X(j,\Delta)$
  and for all $n\leq N$ by
  \begin{align}
    \label{eq:Lambda_def}
    \mathcal{T}(\phi)(t_{n})
    =\left(\displaystyle\prod_{l=1}^nS_{\delta t,K,l,j} \right)\phi_0
    -i\lambda\delta t\displaystyle\sum_{m=1}^n\left(\displaystyle\prod_{l=m+1}^n S_{\delta t,K,l,j}\right)P_KT_{\delta t,K,m,j}^{-1}g^k_L(\phi(t_{m-1}),\phi(t_m)).
  \end{align}
  We can show using Lemma \ref{lem:stability_lin}, that for all $\phi,\psi\in X(j,\Delta)$,
  \begin{align}
    \norm{X(j,\Delta)}{\mathcal{T}\phi}^2&\leq 2 e^{C(j,K_0)\Delta}\normSigma{j}{\phi_0}^2+4\Delta^2C(j,L)e^{C(j,K_0)\Delta}\norm{X(j,\Delta)}{\phi}^2\label{eq:T_ball_disc}\\
    \norm{X(k,\Delta)}{\mathcal{T}(\phi-\psi)}^2&\leq 4\Delta^2C(k,L)e^{C(k,K_0)\Delta}\norm{X(k,\Delta)}{\phi-\psi}^2\label{eq:T_contraction_disc}.
  \end{align}
  The second estimate is established in $X(k,\Delta)$ since we need $g_L^k$ to be Lipschitz.
  Choosing $\Delta$ such that $4\Delta^2C(j,L)e^{C(j,K_0)\Delta}\leq 1/2$ implies that
  $\mathcal{T}$ is a strict contraction in a ball $B^{X(j,\Delta)}_M$ of $X(j,\Delta)$
  defined by $B^{X(j,\Delta)}_M:=\{\phi\in X(j,\Delta),\norm{X(j,\Delta)}{\phi}^2\leq M\}$
  where we chose $M:=4e^{C(j,K_0)\Delta}\norm{X(j,\Delta)}{\phi_0}^2$.
  This allows to show local existence and uniqueness in $X(j,\cdot)$.
  The key point is that the choice of $\Delta$ does not depend on $\phi_0$.
  Thus, we can then iterate this process in $[\Delta,2\Delta]$ in the space
  $L^2_\omega L^\infty(\Delta,2\Delta;\Sigma^{j}(\R^d))$, and so on
  to get existence and uniqueness of a solution $X(j,T)$.
  Moreover since at each iteration the radius $M$ is multiplied by $4e^{C(j,K_0)\Delta}$,
  we eventually get
  after repeating this construction on $\frac{T}{\Delta}$ intervals
  that,
  \begin{align*}
    \esp{\sup_{t\leq T} \normSigma{j}{\phi(t)}^2}\leq 4^{T/\Delta}e^{C(j,K_0) T} \normSigma{j}{\phi_0}^2.
  \end{align*}
\end{proof}
\vspace{3mm}

\begin{proof}[Proof of Lemma \ref{lem:stability_lin}]
  To simplify the notations, we set $\phi^{n+1}=S_{\delta t,K,n+1,j}\phi^n$.
  We set $\widetilde{\chi}^{n+1}$ the random variable, independent of $\mathcal{F}_n$,
  defined by:
  \begin{equation*}
    \widetilde{\chi}^{n+1}=
    \left\{
      \begin{aligned}
        &\widetilde{\chi}^{n+1}\quad\text{if}\quad\sqrt{\delta t}\abs{\widetilde{\chi}^{n+1}}\leq C_0(j),\\
        &0\quad\text{otherwise}.
      \end{aligned}
    \right.
  \end{equation*}

  Thanks to the symmetry of $B_K$, one can write for almost every $\omega$,
  \begin{align}
    \label{eq:lemma_stab_lin_eq1}
    \normSigma{j}{\phi^{n+1}}^2=\normSigma{j}{\phi^n}^2+\dfrac{\sqrt{\delta t}}{4}\widetilde{\chi}^{n+1}\Im \langle [A^j,B_K](\phi^{n+1}+\phi^n),(\phi^{n+1}+\phi^n) \rangle.
  \end{align}
  We notice then,
  \begin{align}
    \label{eq:ito_discrete_linear}
    \sqrt{\delta t}\widetilde{\chi}^{n+1}\Im \langle [A^j,B_K](\phi^{n+1}+\phi^n),(\phi^{n+1}+\phi^n) \rangle=I_{n+1}+II_{n+1}+III_{n+1},
  \end{align}
  with
  \begin{align*}
    I_{n+1}=&-\sqrt{\delta t}\widetilde{\chi}^{n+1}\Im \langle [A^j,B_K](\phi^{n+1}-\phi^n),(\phi^{n+1}-\phi^n) \rangle,\\
    II_{n+1}=&2\sqrt{\delta t}\widetilde{\chi}^{n+1}\Im \langle [A^j,B_K](\phi^{n+1}+\phi^n),(\phi^{n+1}-\phi^n) \rangle,\\
    III_{n+1}=&4\sqrt{\delta t}\widetilde{\chi}^{n+1}\Im \langle [A^j,B_K]\phi^n,\phi^n \rangle.
  \end{align*}

  Therefore, by setting $S_{p+1}=\displaystyle\sum_{n=0}^p III_{n+1}$,
  Equation \eqref{eq:lemma_stab_lin_eq1} can be rewritten
  \begin{align}
    \label{eq:lemma_stab_lin_eq2}
    \normSigma{j}{\phi^{p+1}}^2=
    \normSigma{j}{\phi^0}^2
    +S_{p+1}
    +\displaystyle\sum_{n=0}^p (I_{n+1}+II_{n+1}).
  \end{align}

  We begin with giving some estimates on $I_n$, $II_n$ and $III_n$,
  \begin{lem}
    \label{lem:stab_lin_first_estimates}
    There exists $C(j)>0$, depending only on $j$, such that
    \begin{align*}
      \espc{\abs{I_{n+1}}^2+\abs{II_{n+1}}^2}{\mathcal{F}_n}^{1/2}&\leq
      C(j)\delta t(1+K^2\delta t^{1/2})\normSigma{j}{\phi^n}^2,\\
      \espc{\abs{III_{n+1}}^2}{\mathcal{F}_n}&\leq
      C(j)\delta t\normSigma{j}{\phi^n}^4,\\
      \espc{III_{n+1}}{\mathcal{F}_n}&=0.
    \end{align*}
  \end{lem}
  The proof of this technical Lemma is postponed to the end of this proof
  of Lemma \ref{lem:stability_lin}.

  We show now an estimation on $\esp{\normSigma{j}{\phi^{n+1}}^4}$ with respect to
  $\esp{\normSigma{j}{\phi^{n}}^4}$.
  First, equation \eqref{eq:lemma_stab_lin_eq1} gives
  \begin{align}
    \label{eq:stability_lin_proof1}
    \esp{\normSigma{j}{\phi^{n+1}}^4}=
    &\esp{\normSigma{j}{\phi^{n}}^4}
    +\frac{1}{16}\esp{\left(I_{n+1}+II_{n+1}+III_{n+1}\right)^2}\\
    +&\frac{1}{2}\esp{\normSigma{j}{\phi^{n}}^2\left(I_{n+1}+II_{n+1}+III_{n+1}\right)}.
  \end{align}

  Then, by using Lemma \ref{lem:stab_lin_first_estimates}, we obtain the following
  estimations for $\delta t\leq 1$,
  \begin{align}
    \label{eq:stability_lin_proof2}
    \esp{\left(I_{n+1}+II_{n+1}+III_{n+1}\right)^2}\leq
    C(j) \delta t (1+K^2\delta t^{1/2})^2 \esp{\normSigma{j}{\phi^{n}}^4},
  \end{align}

  \begin{equation}
    \label{eq:stability_lin_proof3}
    \begin{aligned}
      \esp{\normSigma{j}{\phi^{n}}^2\left(I_{n+1}+II_{n+1}+III_{n+1}\right)}
      =&\esp{\normSigma{j}{\phi^{n}}^2\left(I_{n+1}+II_{n+1}\right)}\\
      \leq&\esp{\normSigma{j}{\phi^{n}}^2 \espc{(I_{n+1}+II_{n+1})^2}{\mathcal{F}_n}^{1/2}}\\
      \leq& C(j) \delta t(1+K^2\delta t^{1/2}) \esp{\normSigma{j}{\phi^{n}}^4}.
    \end{aligned}
  \end{equation}

  Collecting \eqref{eq:stability_lin_proof1}, \eqref{eq:stability_lin_proof2}
  and \eqref{eq:stability_lin_proof3}, we obtain for $\delta t\leq 1$,
  \begin{align*}
    \esp{\normSigma{j}{\phi^{n+1}}^4}
    \leq C(j)\delta t(1+K^2\delta t^{1/2})^2\esp{\normSigma{j}{\phi^{n}}^4}.
  \end{align*}
  Then by using Gronwall's inequality, we can estimate $\esp{\normSigma{j}{\phi^{n}}^4}$
  uniformly in $n\leq N$ and $\delta t\leq 1$,
  \begin{align}
      \label{eq:lemma_stab_lin_eq3}
    \esp{\normSigma{j}{\phi^n}^4}^{1/2}\leq e^{C(j)T(1+K^2\delta t^{1/2})^2}\normSigma{j}{\phi_0}^2.
  \end{align}

  We now turn to the estimate of $\esp{\sup_{p\leq N}\normSigma{j}{\phi^p}^2}$.
  We proceed by triangular inequality by bounding first
  $\esp{\sup_{p\leq N}S_p}$ and then
  $\esp{\sup_{p\leq N} \sum_{n=0}^p (II_{n+1}+III_{n+1})}$
  in Equation \eqref{eq:lemma_stab_lin_eq2}.
  Noticing that $(S_p)_{p\in\N}$ is a $(\mathcal{F}_p)$-martingale in $L^2$,
  and using Doob's inequality,
  \begin{align*}
    \esp{\displaystyle\sup_{p\leq N}S_p}\leq 2\esp{S_N^2}^{1/2}.
  \end{align*}
  Since for all $n\in\N$ $\espc{III_{n+1}}{\mathcal{F}_n}=0$, using Lemma
  \ref{lem:stab_lin_first_estimates} and Equation \eqref{eq:lemma_stab_lin_eq3},
  \begin{align*}
    \esp{S_N^2}
    &=\displaystyle\sum_{n=0}^{N-1}\esp{III_{n+1}^2}\\
    &\leq C(j)T e^{C(j)T(1+K^2\delta t^{1/2})^2}\normSigma{j}{\phi_0}^4.
  \end{align*}
  Then,
  \begin{align}
    \label{eq:lemma_stab_lin_eq4}
    \esp{\displaystyle\sup_{p\leq N}S_p}\leq C(j)T e^{C(j)T(1+K^2\delta t^{1/2})}\normSigma{j}{\phi_0}^2.
  \end{align}

  The estimate of $\esp{\sup_{p\leq N} \sum_{n=0}^p (II_{n+1}+III_{n+1})}$
  is done using triangular inequality, Lemma \ref{lem:stab_lin_first_estimates} and
  Equation \eqref{eq:lemma_stab_lin_eq3},
  \begin{align*}
    \esp{\displaystyle\sup_{p\leq N} \displaystyle\sum_{n=0}^p II_{n+1}+III_{n+1}}^2
    &\leq \esp{\left(\displaystyle\sup_{p\leq N} \displaystyle\sum_{n=0}^p II_{n+1}+III_{n+1}\right)^2}\\
    &\leq N \esp{\displaystyle\sum_{n=0}^N (II_{n+1}+III_{n+1})^2},
  \end{align*}
  and using Lemma \ref{lem:stab_lin_first_estimates} and the uniform bound given by
  \eqref{eq:lemma_stab_lin_eq3},
  \begin{align*}
    \esp{\displaystyle\sup_{p\leq N} \displaystyle\sum_{n=0}^p II_{n+1}+III_{n+1}}^2
    \leq C(j) N^2 \delta t^2 (1+K^2\delta t^{1/2})^2 e^{C(j)T(1+K^2\delta t^{1/2})^2}\normSigma{j}{\phi_0}^2,
  \end{align*}
  that is to say,
  \begin{align}
    \label{eq:lemma_stab_lin_eq5}
    \esp{\displaystyle\sup_{p\leq N} \displaystyle\sum_{n=0}^p II_{n+1}+III_{n+1}}
    \leq C(j,T,K^4\delta t) \normSigma{j}{\phi_0}^2.
  \end{align}

  Collecting Equations \eqref{eq:lemma_stab_lin_eq2}, \eqref{eq:lemma_stab_lin_eq4} and \eqref{eq:lemma_stab_lin_eq5},
  we obtain that for $\delta t\leq 1$, there exists $C(j,T,K^4\delta t)$ such that
  \begin{align*}
    \esp{\sup_{n\leq N}\normSigma{j}{\phi^n}^2}\leq C(j,T,K^4\delta t)\normSigma{j}{\phi_0}^2.
  \end{align*}
\end{proof}

We prove now the technical Lemma \ref{lem:stab_lin_first_estimates} introduced in
the proof of the Lemma \ref{lem:stability_lin}.
\begin{proof}[Proof of Lemma \ref{lem:stab_lin_first_estimates}]
  We prove estimates on $I_{n+1}$ and $II_{n+1}$.
  Equation \eqref{as:eq:3} enables us to estimate the inner product that appears
  in $I_{n+1}$,
  \begin{align*}
    \abs{I_{n+1}}=&
    \abs{\sqrt{\delta t}\widetilde{\chi}^{n+1}\Im \langle [A^j,B_K](\phi^{n+1}-\phi^n),
    (\phi^{n+1}-\phi^n) \rangle}\\
    \leq& C(j)\sqrt{\delta t}\abs{\widetilde{\chi}^{n+1}}\normSigma{j}{\phi^{n+1}-\phi^n}^2,
  \end{align*}
  and using again the implicit relation \eqref{eq:C_0_proof}, we obtain
  \begin{align*}
    \abs{I_{n+1}}\leq C(j)\sqrt{\delta t}\abs{\widetilde{\chi}^{n+1}}\left(\normSigma{j}{\delta t A_K\phi^{n+1/2}}^2+\normSigma{j}{\sqrt{\delta t}\widetilde{\chi}^{n+1} B_K\phi^{n+1/2}}^2\right).
  \end{align*}

  Then, equation \eqref{as:eq:2}, with the fact that
  $\normSigma{j}{\phi^{n+1/2}}^2\leq \sqrt{2}\normSigma{j}{\phi^n}^2$ yields for $\delta t\leq 1$
  \begin{align*}
    \espc{I_{n+1}^2}{\mathcal{F}_n}^{1/2} \leq& C(j)\delta t(1+K^2\delta t^{1/2})\normSigma{j}{\phi^{n}}^2.
  \end{align*}

  We prove now estimate on $II_{n+1}$.
  We recall,
  \begin{align*}
    II_{n+1}=&2\sqrt{\delta t}\widetilde{\chi}^{n+1}\Im \langle [A^j,B_K](\phi^{n+1}+\phi^n),(\phi^{n+1}-\phi^n) \rangle.
  \end{align*}
  Again, we replace $\phi^{n+1}-\phi^n$ by the implicit formulation given by \eqref{eq:C_0_proof},
  \begin{align*}
    II_{n+1}
    =&\delta t^{3/2}\widetilde{\chi}^{n+1}\Re \langle [A^j,B_K](\phi^{n+1}+\phi^n),A_K(\phi^{n+1}+\phi^n)\rangle\\
    &+\delta t(\widetilde{\chi}^{n+1})^2\Re \langle [A^j,B_K](\phi^{n+1}+\phi^n),B_K(\phi^{n+1}+\phi^n)\rangle\\
    =&i+ii.
  \end{align*}
  Then, by \eqref{as:eq:3}
  \begin{align*}
    \abs{i}&\leq C(j) \abs{\widetilde{\chi}^{n+1}} \delta t^{3/2}\normSigma{j}{\phi^{n+1/2}}\normSigma{j+2}{P_K\phi^{n+1/2}}.
  \end{align*}
  Using \eqref{eq:diff_trunc} and Lemma \ref{lem:C_0}, we get,
  \begin{align*}
    \normSigma{j}{\phi^{n+1/2}} \normSigma{j+2}{P_K \phi^{n+1/2}} \leq C K \normSigma{j}{\phi^{n}}^2,
  \end{align*}
  which leads to,
   \begin{align*}
    \abs{i}
    \leq C(j) \delta t^{3/2} K \abs{\widetilde{\chi}^{n+1}}\normSigma{j}{\phi^n}^2,
   \end{align*}
   and it follows that,
  \begin{align}
    \label{eq:II-i}
    \espc{\abs{i}^2}{\mathcal{F}_n}^{1/2}\leq& C(j)\delta t^{3/2}K\normSigma{j}{\phi^n}^2.
  \end{align}
  The estimation of $ii$ follows from \eqref{as:eq:4} and Lemma \ref{lem:C_0},
  \begin{align}
    \label{eq:II-ii}
    \espc{\abs{ii}^2}{\mathcal{F}_n}^{1/2}\leq C(j) \delta t (\chi^{n+1})^2 \normSigma{j}{\phi^n}^2.\\
  \end{align}
  Eventually, collecting \eqref{eq:II-i} and \eqref{eq:II-ii} we get that
  \begin{align*}
    \espc{II_{n+1}^2}{\mathcal{F}_n}^{1/2}\leq C(j)\delta t(1+K^2\delta t^{1/2})\normSigma{j}{\phi^n}^2.
  \end{align*}
\end{proof}

\section{Convergence of the numerical scheme}
\label{sec:conv}
The main objective of this section is to prove that the solution of equation
\eqref{eq:scheme} converges to the solution of equation \eqref{eq:GP} in a probabilistic
sense that we define later.
Our numerical scheme $(\phi^n_{K,L,\delta t})_{n\leq N}$ \eqref{eq:scheme} is
parametrized by the three variables $K$, $L$ and $\delta t$.
Morally, we expect $\phi^n_{K,L,\delta t}$ to tend to $\phi(t_n)$ as we make
$\delta t$ tend to zero, and $K$ and $L$ to infinity.
The goal is now to prove this limit.
To do so, we specify the divergence of $K$ and $L$ with respect to $\delta t$.
We consider them as functions of the time step, and we denote them
by $K(\delta t)$ and $L(\delta t)$.
We are able to specify explicitly the function $K(\delta t)$, but not
the function $L(\delta t)$.
Indeed, we recall that on one hand, in order to prove stability, we had to impose
a maximal growth on $K$ with respect to $\delta t^{-1/4}$.
Now, we actually require it to grow at this speed $\delta t^{-1/4}$
(in other words we suppose that there exists $K_0,K_1>0$ such that
$K_1\leq K(\delta t)\delta t^{-1/4}\leq K_0$ for all $\delta t$).
This condition on $K(\delta t)$ is sufficient to prove the convergence.
Nevertheless Theorem \ref{thm:conv_proba} only establishes the existence
of a choice of $L(\delta t)$ for all $\delta t$ that enables to prove the convergence.
Since we do not track explicitly the dependences on $L$ in our computations,
we are not able to define explicitly this function $L(\delta t)$.
This function is chosen according to the following criteria.
First, for $\delta t$ small enough, $L(\delta t)$ must be small enough
to ensure existence of a solution for the numerical scheme.
This remark defines a maximal speed of divergence on $L(\delta t)$.
Then, since we compute all the error bounds in mean-square up to a multiplicative
constant that depends on $L$, and increasing with $L$, in Proposition
\ref{prop:conv_ms} and \ref{prop:conv_ms2}, then this parameter should
diverge slowly enough with respect to $\delta t$ to still ensure the convergences
stated by these propositions.
This last point is explained in detail in the proof of Theorem \ref{thm:conv_proba}.
\vspace{3mm}

We can now state that our convergence result of
$(\phi^n_{K(\delta t),L(\delta t),\delta t})_{n\in\N}$
toward $(\phi(t_n))_{n\in\N}$, when $\delta t$ tends to zero.
\begin{thm}
  \label{thm:conv_proba}
  We suppose Assumption \ref{as:1} and \ref{as:convergence} to be verified.
  For all $T>0$,
  $k\in\N^*$,
  $\phi_0\in\Sigma^{k+12}(\R^d)$,
  $K_0>K_1>0$,
  $C>0$,
  and all $\alpha<1$,
  there exists a choice of $L(\delta t)$ such that
  \begin{align}
    \label{eq:conv_prob}
    \displaystyle\lim_{\delta t\rightarrow 0}\sup_{n\delta t\leq T}\prob{\normSigma{k}{\phi_{K(\delta t),L(\delta t),\delta t}^n-\phi(t_n)} \geq C\delta t^\alpha}=0.
  \end{align}
\end{thm}
This theorem ensures that the numerical error is small for small $\delta t$ on large sets of $\Omega$.
This kind of convergence result has been studied in \cite{Chen2015} for a numerical scheme
for the stochastic cubic Schr\"{o}dinger equation with multiplicative noise of Stratonovich type.
\vspace{3mm}

In order to explain how this result is obtained, we introduce some notations.
We consider the following truncated and projected equation for a fixed $L>0$.
\begin{equation}
  \label{eq:GP_modif}
  \left\{
    \begin{aligned}
      &id\phi_{K(\delta t),L} -\lambda P_{K(\delta t)} f^k_L(\phi_{K(\delta t),L})dt
      -A_{K(\delta t)}\phi_{K(\delta t),L}dt=B_{K(\delta t)}\phi_{K(\delta t),L}\circ dW_t,\\
      &\phi_{K(\delta t),L}(0)=P_{K(\delta t)}\phi_0.
    \end{aligned}
  \right.
\end{equation}
We denote by $(\phi_{K(\delta t),L}(t))_{t\leq T}$, the solution of this equation.

We also set $(\phi_{L}(t))_{t\leq T}$, the solution of the following truncated equation
\begin{equation}
  \label{eq:GP_trunc}
  \left\{
    \begin{aligned}
      &id\phi_{L} -\lambda f^k_L(\phi_{L})dt - A \phi_{L}dt=|x|^2\phi_{L}\circ dW_t,\\
      &\phi_{L}(0)=\phi_0.
    \end{aligned}
  \right.
\end{equation}
We set now, the $\mathcal{F}_n$-measurable process
$(\psi_{K(\delta t),L,\delta t}^n)_{n\leq N}$ which is essentially the
stopped version of $(\phi_{K(\delta t),L}(t))_{t\leq T}$
at stopping time $\tau(\delta t)$, defined by
\begin{equation}
  \label{eq:psi-stopped}
  \psi_{K(\delta t),L,\delta t}^n=\left\{
  \begin{aligned}
    &P_K\phi_0,\quad\text{if}\quad n=0,\\
    &\phi_{K(\delta t),L}(t_n),\quad\text{if}\quad 0<n<\tau(\delta t),\\
    &\psi_{K(\delta t),L,\delta t}^{n-1},\quad\text{otherwise}.
  \end{aligned}
  \right.
\end{equation}
\vspace{3mm}

The convergence is obtained in several steps.
We start by proving that $(\phi^n_{K(\delta t),L,\delta t})$ can be
as close as we want (for the $L^\infty L^2_\omega \Sigma^k$ norm)
to $(\psi^n_{K(\delta t),L,\delta t})$ by choosing $\delta t$
small enough.
This result is stated in Proposition \ref{prop:conv_ms}.
Then we use the fact that $(\psi^n_{K(\delta t),L,\delta t})$ is equal to
$(\phi_{K(\delta t),L}(t_n))$ with high probability when $\delta t$ is small.
We use then the fact that $(\phi_{K(\delta t),L}(t_n))$ converges to
$(\phi_{L}(t_n))$ for the $L^\infty L^2_\omega \Sigma^k$ norm at order one in $\delta t$,
provided that $K(\delta t)$ grows at least at speed $\delta t^{-1/4}$.
This result is stated in Proposition \ref{prop:conv_ms2}.
Eventually, we use the fact that $(\phi_{L}(t_n))$ is equal to $(\phi(t_n))$
with high probability when $L$ is large.

\begin{prop}
  \label{prop:conv_ms}
  We suppose Assumption \ref{as:1} to be verified.
  For all $T>0$,
  for all $k\in\N^*$,
  for all $L\in\N^*$
  and for all $K_0>0$,
  there exists $N_0(T,k,L)\in\N^*$ and $C(T,k,L,K_0)>0$, such that,
  for all $N\geq N_0(T,k,L)$ and $\delta t=T/N$,
  for all $K\leq K_0 \delta t^{-1/4}$,
  and for all $\phi_0\in\Sigma^{k+12}(\R^d)$,
  \begin{align}
    \label{eq:conv_ms}
    \sup_{n\leq N}\esp{\normSigma{k}{\phi_{K(\delta t),L,\delta t}^n-\psi_{K(\delta t),L,\delta t}^n}^2}
    \leq C(T,k,L,K_0)\delta t^2\normSigma{k+12}{\phi_0}^2,
  \end{align}
  where $(\phi_{K(\delta t),L,\delta t}^n)_{n\leq N}$ is defined by \eqref{eq:scheme},
  and $(\psi_{K(\delta t),L,\delta t}^n)_{n\leq N}$ is defined by \eqref{eq:psi-stopped}.
\end{prop}
\begin{prop}
  \label{prop:conv_ms2}
  We suppose Assumption \ref{as:1} and \ref{as:convergence} to be verified.
  For all $T>0$,
  for all $k\in\N^*$,
  for all $L\in\N^*$
  and for all $K_1>0$,
  there exists $C(T,k,L,K_1)>0$, such that,
  for all $\delta t=T/N$,
  for all $K\geq K_1 \delta t^{-1/4}$,
  and for all $\phi_0\in\Sigma^{k+12}(\R^d)$,
  \begin{align}
    \label{eq:conv_ms2}
    \sup_{n\leq N}\esp{\normSigma{k}{\phi_{K(\delta t),L}(t_n)-\phi_{L}(t_n)}^2}
    \leq C(T,k,L,K_1)\delta t^2\normSigma{k+12}{\phi_0}^2.
  \end{align}
\end{prop}
The proof of Proposition \ref{prop:conv_ms} is quite technical.
It is done by controlling the numerical error at one step by the error at the previous
step, in order to use Gronwall's inequality to conclude.
This proof uses extensively the boundedness of operators $B_K$ as in
Lemma \ref{lem:stability_lin}.
The proof of Proposition \ref{prop:conv_ms2} relies on Itô's lemma,
Assumption \ref{as:convergence} and the fact that $K$ grows as speed $\delta t^{-1/4}$.
\vspace{3mm}

We prove now first Theorem \ref{thm:conv_proba}, and then Propositions \ref{prop:conv_ms}
and \ref{prop:conv_ms2}.
The proof of Theorem \ref{thm:conv_proba} relies on these two propositions.
We choose to show this proof first since the arguments are quite simple and follow on
from the convergence results of Propositions \ref{prop:conv_ms} and \ref{prop:conv_ms2}.
\begin{proof}[Proof of Theorem \ref{thm:conv_proba}]

  Our goal is to show now that for all $\epsilon>0$ there exists a pair $(\delta t_1,L)$
  such that $\delta t_1\leq \delta t_0(j,L,K_0)$ and for all $\delta t\leq \delta t_1$,
  \begin{align*}
    \prob{\normSigma{k}{\phi_{K(\delta t),L,\delta t}^n-\phi(t_n)} \geq C\delta t^\alpha}
    \leq \epsilon.
  \end{align*}
  Then, we can set $L(\delta t_1)$ to be such a $L$ associated with $\delta t_1$.

  For $L>0$ and $\delta t\leq \delta t_0(j,L,K_0)$, we upper bound $\prob{\normSigma{k}{\phi_{K(\delta t),L,\delta t}^n-\phi(t_n)\geq C\delta t^\alpha}}$
  in the following way,
  \begin{align*}
      \prob{\normSigma{k}{\phi_{K(\delta t),L,\delta t}^n-\phi(t_n)} \geq C\delta t^\alpha}
      \leq
      &\prob{\normSigma{k}{\phi_{K(\delta t),L,\delta t}^n-\psi_{K(\delta t),L,\delta t}^n} \geq C\delta t^\alpha/4}\\
      &+\prob{\normSigma{k}{\psi_{K(\delta t),L,\delta t}^n-\phi_{K(\delta t),L}(t_n)} \geq C\delta t^\alpha/4}\\
      &+\prob{\normSigma{k}{\phi_{K(\delta t),L}(t_n)-\phi_{L}(t_n)} \geq C\delta t^\alpha/4}\\
      &+\prob{\normSigma{k}{\phi_{L}(t_n)-\phi(t_n)} \geq C\delta t^\alpha/4}\\
      \leq& i+ii+iii+iv.
  \end{align*}

  Choosing $L$ large enough ensures that
  $\prob{\sup_{n\leq N}\normSigma{k}{\phi(t_n)}\geq L}\leq \epsilon/4$
  since $(\phi(t))_{t\geq 0}\in\mathcal{C}(\R^+,\Sigma^k)$.
  Then it comes $iv\leq \epsilon/4$.

  At the cost of supposing $\delta t$ even smaller, we can assume that
  $\prob{\tau(\delta t)\leq N}\leq \epsilon/4$, which leads to $ii\leq \epsilon/4$.
  This follows from Lemma \ref{lem:tau_conv}.

  We now use Proposition \ref{prop:conv_ms} and \ref{prop:conv_ms2} with Markov's inequality to choose
  $C$ such that $i+iii\leq \epsilon/2$:
  \begin{align*}
    i+iii\leq&
    \dfrac{1}{C^2\delta t^{2\alpha}}
    \left(
    \esp{\normSigma{k}{\phi_{K(\delta t),L,\delta t}^n-\psi_{K(\delta t),L,\delta t}^n}^2}
    +\esp{\normSigma{k}{\phi_{K(\delta t),L}(t_n)-\phi_{L}(t_n)}^2}
    \right)\\
    \leq&
    \dfrac{\delta t^{2(1-\alpha)}}{C^2}\normSigma{k+12}{\phi_0}^2\left(C(T,k,L,K_0)+C(T,k,L,K_1)\right).
  \end{align*}
  Since $\delta t^{2(1-\alpha)}$ tends to zero when $\delta t$ tends to zero,
  then it is possible to take $\delta t$ even smaller to ensure that $i+iii\leq \epsilon/2$.

  It is clear that for all $\epsilon$ we can construct a pair $(\delta t_1,L(\delta t_1))$
  with $\delta t_1$ that tends to zero when $\epsilon$ tends to zero, and such that
  \begin{align*}
    \prob{\normSigma{k}{\phi_{K(\delta t_1),L(\delta t_1),\delta t_1}^n-\phi(t_n)} \geq C\delta t_1^\alpha}
    \leq \epsilon.
  \end{align*}
  It implies the result,
  \begin{align*}
    \lim_{\delta t\to 0} \prob{\normSigma{k}{\phi_{K(\delta t),L(\delta t),\delta t}^n-\phi(t_n)} \geq C\delta t^\alpha}
    = 0.
  \end{align*}
\end{proof}
\vspace{3mm}

\begin{proof}[Proof of Proposition \ref{prop:conv_ms}]

To simplify notations, we get rid of the subscripts $K(\delta t)$, $L$ and $\delta t$
in this step.
For example, we denote $\phi^n_{K(\delta t),L,\delta t}$ by $\phi^n$.
We also use $g$ and $f$ instead of $g_L^k$ and $f_L^k$ defined by \eqref{eq:f_L}
and \eqref{eq:f_L2}.
We set for all $n\leq N$, $e^n=\phi^n-\psi^n$.

Our goal is to show that there exists $C(T,k,L,K_0)>0$, such that,
\begin{align}
  \label{eq:goal_th1}
  \esp{\normSigma{k}{e^{n+1}}^2}
  \leq (1+\delta t C(T,k,L,K_0))\esp{\normSigma{k}{e^{n}}^2}
  +\delta t^3 C(T,k,L,K_0)\normSigma{k+12}{\phi_0}^2.
\end{align}
This will enable us to conclude this proposition by using the discrete Gronwall's Lemma.
To show Equation \eqref{eq:goal_th1}, we are going to show that,
\begin{align}
  \label{eq:goal_th1bis}
  \abs{\Re\esp{\langle e^{n+1}-e^n,A^k e^{n+1/2}\rangle}}
  \leq C(T,k,L,K_0)\delta t\left(\esp{\normSigma{k}{e^{n}}^2}
  +\delta t^2\normSigma{k}{\phi_0}^2\right).
\end{align}
To prove these estimation, we begin with splitting $e^{n+1}-e^n$ as
a sum of several terms, to make clear how each of these terms are small.

Using Equations \eqref{eq:psi-stopped}, \eqref{eq:GP_modif} and \eqref{eq:scheme},
we get that
\begin{align*}
  \phi^{n+1}=\phi^n + \indic{\tau(\delta t)>n+1}\left(
  -i\delta t A_K \phi^{n+1/2}
  -i\sqrt{\delta t}\chi^{n+1} B_K \phi^{n+1/2}
  -i\delta t P_K g(\pi^{n+1},\phi^n)
  \right),
\end{align*}
and
\begin{align*}
  \psi(t_{n+1})=\psi(t_{n})+ \indic{\tau(\delta t)>n+1}\left(
  -i\int_{t_n}^{t_{n+1}} A_K \phi(s)ds
  -i\int_{t_n}^{t_{n+1}} B_K \phi(s)\circ dW_s
  -i\int_{t_n}^{t_{n+1}} P_K f(\phi(s))ds
  \right).
\end{align*}
We split $e^{n+1}-e^n$ in the following way,
\begin{align}
  \label{eq:decomp_e}
    e^{n+1}=e^n+\indic{\tau(\delta t)>n+1}\left(a_1+a_2+a_3\right),
\end{align}
with
\begin{align}
  \label{eq:e_n+1}
  a_1&=-i\delta tA_K\phi^{n+1/2}+i\displaystyle\int_{t_n}^{t_{n+1}}A_K\phi(s)ds,\\
  a_2&=-i\sqrt{\delta t}\chi^{n+1}B_K\phi^{n+1/2}+i\displaystyle\int_{t_n}^{t_{n+1}}B_K\phi(s)\circ dW_s,\\
  a_3&=-i\delta t P_K g(\phi^{n+1},\phi^{n})+i\displaystyle\int_{t_n}^{t_{n+1}} P_K f(\phi(s))ds.\\
\end{align}
$a_1$ comes from the term $A_K\phi_L dt$ in \eqref{eq:GP_trunc},
$a_2$ comes from the term $B_K\phi_L\circ dW_t$,
and $a_3$ comes from the nonlinear term.
We now split these terms again to make clear how we are going to prove their \emph{smallness}.
We set $a_1=a_{1,1}+a_{1,2}+a_{1,3}$ with,
\begin{align*}
  a_{1,1}&=-i\delta t A_K e^{n+1/2},\\
  a_{1,2}&=-\dfrac{i}{2}\delta t A_K (\phi(t_{n+1})-\phi(t_n)),\\
  a_{1,3}&=i\displaystyle\int_{t_n}^{t_{n+1}}A_K(\phi(s)-\phi(t_n))ds.
\end{align*}
We define $R(t,s)$ for $t\geq s$ by
\begin{align}
  \label{eq:Rts}
R(t,s)=\displaystyle\int_{s}^{t}A_K\phi(r)dr
+\dfrac{1}{2}\displaystyle\int_{s}^{t}B_K^2\phi(r)dr
+\displaystyle\int_{s}^{t}P_K f(\phi(r))dr.
\end{align}
$a_2$ can be split in the following way.
\begin{align*}
  a_2=&-i\sqrt{\delta t}\chi^{n+1}B_Ke^{n+1/2}
  -i\dfrac{\sqrt{\delta t}}{2}\chi^{n+1} B_K(\phi(t_{n+1})-\phi(t_n))
  +i\int_{t_n}^{t_{n+1}} B_K (\phi(s)-\phi(t_n)) \circ dW_s\\
  =&-i\sqrt{\delta t}\chi^{n+1}B_Ke^{n+1/2}
  +\dfrac{1}{2} \int_{t_n}^{t_{n+1}} B_K^2 (\phi(s)-\phi(t_n)) ds
  -i\dfrac{\sqrt{\delta t}}{2}\chi^{n+1} B_K(\phi(t_{n+1})-\phi(t_n))\\
  &+i\int_{t_n}^{t_{n+1}} B_K (\phi(s)-\phi(t_n)) dW_s.
\end{align*}
The last equality is just the It\^{o} form of the first one.
We split again the two last terms in the right-hand side of the second equality
by using again Equation
\eqref{eq:GP_modif} to explicit $\phi(t_{n+1})-\phi(t_n)$ and $\phi(s)-\phi(t_n)$.
Then, we can split $a_2$ by setting $a_2=a_{2,1}+a_{2,2}+a_{2,3}+a_{2,4}$ with,
\begin{align*}
  a_{2,1}&=-i\sqrt{\delta t}\chi^{n+1}B_Ke^{n+1/2},\\
  a_{2,2}&=-\dfrac{\sqrt{\delta t}}{2}\chi^{n+1}B_K\displaystyle\int_{t_n}^{t_{n+1}}B_K(\phi(s)-\phi(t_n))dW_s
  +\displaystyle\int_{t_n}^{t_{n+1}}B_K\displaystyle\int_{t_n}^{s}B_K(\phi(r)-\phi(t_n))dW_r dW_s,\\
  a_{2,3}&=\dfrac{1}{2}\displaystyle\int_{t_n}^{t_{n+1}}B_K^2(\phi(s)-\phi(t_n))ds,\\
  a_{2,4}&=-\dfrac{\sqrt{\delta t}}{2}\chi^{n+1}B_KR(t_{n+1},t_n)+\displaystyle\int_{t_n}^{t_{n+1}}B_KR(s,t_n)dW_s.
\end{align*}
We set $a_3=a_{3,1}+a_{3,2}+a_{3,3}$ with,
\begin{align*}
  a_{3,1}&=-i\delta t P_K\left( g(\phi^{n+1},\phi^n)-g(\phi(t_{n+1}),\phi(t_n)) \right),\\
  a_{3,2}&=-i\delta t P_K \left( g(\phi(t_{n+1}),\phi(t_n))-f(\phi(t_n)) \right),\\
  a_{3,3}&=i\displaystyle\int_{t_n}^{t_{n+1}} P_K \left( f(\phi(s))-f(\phi(t_n)) \right)ds.
\end{align*}

Terms $a_{1,2}$, $a_{1,3}$, $a_{2,2}$, $a_{2,3}$, $a_{2,4}$, $a_{3,2}$ and $a_{3,3}$
have similar \emph{smallness} properties. This is why we set $f$ defined by
\begin{align*}
  f=a_{1,2}+a_{1,3}+a_{2,2}+a_{2,3}+a_{2,4}+a_{3,2}+a_{3,3}.
\end{align*}

Then, Equation \eqref{eq:decomp_e} can be written as
\begin{align}
  \label{eq:decomp_e2}
  e^{n+1}-e^n=\indic{\tau(\delta t)>n+1}\left(a_{1,1}+a_{2,1}+a_{3,1}+f\right).
\end{align}
Plugging this expression of $e^{n+1}-e^n$ into Equation \eqref{eq:goal_th1bis}
and using the triangular inequality it becomes clear that,
in order to show the estimation \eqref{eq:goal_th1bis},
it is enough to show that for all $1\leq p\leq 3$,
\begin{align}
  \label{eq:ggoal1}
  \abs{\Re\esp{\indic{\tau(\delta t)>n+1}\langle a_{p,1},A^k e^{n+1/2}\rangle}}
  \leq C(k,j,T)\delta t\left(\esp{\normSigma{k}{e^{n}}^2}+\delta t^2\normSigma{k}{\phi_0}^2\right),
\end{align}
and
\begin{align}
  \label{eq:ggoal2}
  \abs{\Re\esp{\indic{\tau(\delta t)>n+1}\langle f,A^k e^{n+1/2}\rangle}}
  \leq C(k,j,T)\delta t\left(\esp{\normSigma{k}{e^{n}}^2}+\delta t^2\normSigma{k}{\phi_0}^2\right).
\end{align}
If we try instead to take the square of each side of equation \eqref{eq:decomp_e},
then we won't be able to estimate properly $\esp{\normSigma{k}{a_2}^2}$.

Before starting to prove Equations \eqref{eq:ggoal1} and \eqref{eq:ggoal2},
we explain why we chose to decompose $e^{n+1}-e^n$ as in Equation \eqref{eq:decomp_e2}.
As we said, it comes from the \emph{smallness} properties of $f$,
which is quantified by the following Lemma,
\begin{lem}
  \label{lem:conv_lem3}
  There exists $C(L,j,T)>0$ such that,
  \begin{align*}
    \esp{\normSigma{j}{f}^4}^{1/2}\leq C(L,j,T)\delta t^3\normSigma{j+8}{\phi_0}^2.
  \end{align*}
\end{lem}
\begin{proof}[Proof of Lemma \ref{lem:conv_lem3}]
The proof of this Lemma is a direct consequence of Lemmas
\ref{lem:conv_lem1} and \ref{lem:conv_lem2}.
\end{proof}
\begin{lem}
  \label{lem:conv_lem1}
  For all $T>0$, $L\in\N^*$, $K(\delta t)\in\N^*$, $j\in\N$ if $\phi_0\in\Sigma^j(\R^d)$,
  then $\phi_{K(\delta t),L}\in L^{\infty}(0,T;L^4_\omega\Sigma^j(\R^d))$,
  where $\phi_{K(\delta t),L}$ is the solution of \eqref{eq:GP_modif},
  and
  there exists $C(T,L,j)>0$, independent of $\phi_0$, such that,
  \begin{align*}
    \displaystyle\sup_{t\leq T}\esp{\normSigma{j}{\phi_{K(\delta t),L}(t)}^4}^{1/2}
    \leq C(T,L,j) \normSigma{j}{\phi_0}^2.
  \end{align*}
\end{lem}
\begin{lem}
  \label{lem:conv_lem2}
  For all $T>0$, $L\in\N^*$, $j\in\N$, $K(\delta t)\in\N^*$ and for all $\phi_0\in\Sigma^j(\R^d)$,
  there exists $C(T,L,j)>0$, independent of $\phi_0$, such that, for all $s\leq t$
  \begin{align*}
    \esp{\normSigma{j}{\phi_{K(\delta t),L}(t)-\phi_{K(\delta t),L}(s)}^4}^{1/2}
    \leq C(L,j) (t-s) \normSigma{j+4}{\phi_0}^2,
  \end{align*}
  and
  \begin{align*}
    \esp{\normSigma{j}{R(t,s)}^4}^{1/2}
    \leq C(L,j) (t-s)^2 \normSigma{j+4}{\phi_0}^2,
  \end{align*}
  where $\phi_{K(\delta t),L}$ is the solution of \eqref{eq:GP_trunc}.
  and $R(t,s)$ is defined by \eqref{eq:Rts}.
\end{lem}
Their proof is let to the reader and follows from It\^{o}'s Lemma.

Yet, in the following of the proof, we need another stronger result,
\begin{lem}
  \label{lem:conv_lem4}
  $f$ can be split in two terms $f_1$ and $f_2$,
  such that there exists $C(L,j,T)>0$ for which
  \begin{align}
    \label{eq:f1_est}
    \esp{\normSigma{j}{f_1}^4}^{1/2}&\leq C(L,j,T)\delta t^4\normSigma{j+8}{\phi_0}^2,
  \end{align}
  \begin{align*}
    \esp{\normSigma{j}{f_2}^4}^{1/2}&\leq C(L,j,T)\delta t^3\normSigma{j+8}{\phi_0}^2,
  \end{align*}
  and
  \begin{align*}
    \esp{\langle f_2,A^k e_n \rangle}=0.
  \end{align*}
\end{lem}
\begin{proof}[Proof of Lemma \ref{lem:conv_lem4}]
  The proof of Lemma \ref{lem:conv_lem4} is quite computational, but not difficult.
  The idea is to develop the terms $\phi(s)-\phi(t_n)$ that appear in the expressions
  of the $a_{i,j}$ using Equation \eqref{eq:GP_trunc} in the Itô form.
  Then one has to collect the terms with zero expectation when taking their
  $L^2$ inner product with $A^k e_n$.
  These terms form $f_2$.
  The remaining terms form $f_1$ and their estimation \eqref{eq:f1_est} follows
  from a direct application of Lemmas \ref{lem:conv_lem1} and \ref{lem:conv_lem2}.
\end{proof}

We can state the following Lemma, that gives a naive estimation of $e^{n+1}-e^n$,
\begin{lem}
  \label{lem:conv_lem6}
  For all $j\in\N^*$, $L>0$ and $K_0>0$,
  there exists $C(L,j,T,K_0)>0$ such that, for all
  $\delta t\leq \delta t_0(j+2,L,K_0)$
  (where $\delta t_0(j+2,L,K_0)$ is given by Proposition \ref{prop:exist_scheme})
  \begin{align*}
    \esp{\normSigma{j}{e^{n+1}-e^n}^2}
    \leq C(L,j,T,K_0) \left(
      \delta t \esp{\normSigma{j+2}{e^n}^2}
      +\delta t^3 \esp{\normSigma{j+10}{f}^2}
    \right).
  \end{align*}
\end{lem}
\begin{proof}[Proof of Lemma \ref{lem:conv_lem6}]
  The proof starts with the following triangular inequality,
  \begin{align}
    \label{eq:naive0}
    \esp{\normSigma{j}{e^{n+1}-e^n}^2}
    \leq 3 \left(
      \esp{\normSigma{j}{a_{1,1}}^2}
      +\esp{\normSigma{j}{a_{2,1}}^2}
      +\esp{\normSigma{j}{a_{3,1}}^2}
      +\esp{\normSigma{j}{f}^2}
    \right).
  \end{align}
  Then, it comes directly that,
  \begin{align*}
    \esp{\normSigma{j}{a_{1,1}}^2}\leq \delta t^2 \esp{\normSigma{j+2}{e^{n+1/2}}^2},
  \end{align*}
  and using Lemma \ref{lem:conv_lem5},
  \begin{align}
    \label{eq:naive1}
    \esp{\normSigma{j}{a_{1,1}}^2}\leq C(j,L) \delta t^2 \left(
    \esp{\normSigma{j+2}{e^{n}}^2}+\esp{\normSigma{j+2}{f}^2}
    \right).
  \end{align}
  Similarly,
  \begin{align*}
    \esp{\normSigma{j}{a_{2,1}}^2}\leq \delta t \esp{ (\chi^{n+1})^2\normSigma{j+2}{e^{n+1/2}}^2},
  \end{align*}
  and using Lemma \ref{lem:conv_lem5},
  \begin{align}
    \label{eq:naive2}
    \esp{\normSigma{j}{a_{2,1}}^2}\leq C(j,L) \delta t\left(
      \esp{\normSigma{j+2}{e^{n}}^2}+\esp{\normSigma{j+2}{f}^2}
    \right).
  \end{align}
  The third term is estimated from the fact that $g$ is Lipschitz,
  \begin{align*}
    \esp{\normSigma{j}{a_{3,1}}^2}\leq C(j,L) \delta t^2 \left(
      \esp{\normSigma{j}{e^{n+1}}^2}+\esp{\normSigma{j}{e^{n+1}}^2}
    \right),
  \end{align*}
  and using Lemma \ref{lem:conv_lem5} yields,
  \begin{align}
    \label{eq:naive3}
    \esp{\normSigma{j}{a_{3,1}}^2}\leq C(j,L) \delta t^2 \left(
      \esp{\normSigma{j}{e^{n}}^2}+\esp{\normSigma{j}{f}^2}
    \right).
  \end{align}
  Eventually, collecting Equations \eqref{eq:naive0}, \eqref{eq:naive1}, \eqref{eq:naive2}
  and \eqref{eq:naive3}, and using Lemma \ref{lem:conv_lem3} yields,
  \begin{align*}
    \esp{\normSigma{j}{e^{n+1}-e^n}^2}
    \leq C(L,j,T,K_0) \left(
      \delta t \esp{\normSigma{j+2}{e^n}^2}
      +\delta t^3 \esp{\normSigma{j+10}{f}^2}
    \right).
  \end{align*}
\end{proof}

We need a last result that makes use of the truncation of the discretized
Brownian motion at each time step.
\begin{lem}
  \label{lem:conv_lem5}
  For all $j\in\N^*$, $L>0$ and $K_0>0$,
  there exists $\delta t_1(j,L,K_0)\leq \delta t_0(j,L,K_0)$
  (where $\delta t_0(j,L,K_0)$ is given by Proposition \ref{prop:exist_scheme})
  and $C(j,L)>0$ such that for all $\delta t\leq\delta t_1(j,L)$, almost surely,
  \begin{align*}
    \normSigma{k}{e^{n+1}}^2\leq C(j,L) \left(\normSigma{k}{e^n}^2+\normSigma{k}{f}^2 \right).
  \end{align*}
\end{lem}
\begin{proof}[Proof of Lemma \ref{lem:conv_lem5}]
  If $\omega\in\left\{ \tau(\delta t)(\omega)<=n+1 \right\}$, then there is
  nothing to prove.
  Otherwise, by definition $\sqrt{\delta t}\abs{\chi^{n+1}}\leq C_0$ and
  we have by \eqref{eq:decomp_e},
  \begin{align*}
      e^{n+1}=e^n+\left(a_1+a_2+a_3\right),
  \end{align*}
  Then,
  \begin{align*}
    \normSigma{k}{e^{n+1}}^2-\normSigma{k}{e^{n}}^2=
    \Re\langle A^k(e^{n+1}+e^n), a_{1,1}+a_{2,1}+a_{3,1}+f \rangle.
  \end{align*}

  Since $a_{1,1}=-i\delta t A_K e^{n+1/2}$, it comes that
  \begin{align*}
    \Re\langle A^k(e^{n+1}+e^n), a_{1,1}\rangle = 0.
  \end{align*}

  Since $a_{2,1}=-i\sqrt{\delta t}\chi^{n+1}B_Ke^{n+1/2}$,
  \begin{align*}
    \Re\langle A^k(e^{n+1}+e^n), a_{2,1}\rangle
    &=-2\sqrt{\delta t}\chi^{n+1}\Im\langle A^k e^{n+1/2},B_K e^{n+1/2} \rangle\\
    &=\sqrt{\delta t}\chi^{n+1}\Im\langle [A^k,B_K] e^{n+1/2},e^{n+1/2} \rangle.
  \end{align*}
  Then, using \eqref{as:eq:3}, we get in the same way of
  Equation \eqref{eq:trunc_maj} and \eqref{eq:trunc_maj2},
  \begin{align*}
    \Re\langle A^k(e^{n+1}+e^n), a_{2,1}\rangle
    &\leq C(j)\sqrt{\delta t}\abs{\chi} \normSigma{k}{e^{n+1/2}}^2\\
    &\leq C_0 C(j) \normSigma{k}{e^{n+1/2}}^2\\
    &\leq C_0 C(j)\dfrac{1}{2} \left( \normSigma{k}{e^{n+1}}^2 + \normSigma{k}{e^{n}}^2 \right).\\
  \end{align*}
  Note that the constant $C(j)$ is the same here as in Equation \eqref{eq:trunc_maj}.
  Then, following the proof of Lemma \ref{lem:C_0},
  we get that $C_0 C(j)=1/3$.

  Since $a_{3,1}=-i\delta t P_K\left( g(\phi^{n+1},\phi^n)-g(\phi(t_{n+1}),\phi(t_n)) \right)$,
  and using Lemma \ref{lem:lips},
  \begin{align*}
    \Re\langle A^k(e^{n+1}+e^n), a_{3,1}\rangle
    &\leq C(j,L)\delta t \left(\normSigma{k}{e^{n+1}}^2+\normSigma{k}{e^{n}}^2 \right).
  \end{align*}
  And eventually, by the Cauchy-Schwarz and the Young inequalities, for all $\epsilon>0$,
  \begin{align*}
    \Re\langle A^k(e^{n+1}+e^n), f\rangle
    \leq \epsilon\normSigma{k}{e^{n+1}}^2+\dfrac{4}{\epsilon}\normSigma{k}{f}^2.
  \end{align*}

  To sum up, we obtain that,
  \begin{align*}
    \normSigma{k}{e^{n+1}}^2\left(1-\dfrac{C_0 C(j)}{2} +C(j,L)\delta t+\epsilon \right)
    \leq \normSigma{k}{e^{n}}^2\left(1+\dfrac{C_0 C(j)}{2} +C(j,L)\delta t \right)
    + \dfrac{4}{3} \normSigma{k}{f}^2.
  \end{align*}
  Then, by taking $\epsilon$, small enough, and choosing $\delta t$ small enough,
  one can ensure that
  \begin{align*}
    \left(1-\dfrac{C_0 C(j)}{2} +C(j,L)\delta t+\epsilon \right)>0,
  \end{align*}
  which proves that there exists $C(j,L)>0$, such that,
  \begin{align}
    \normSigma{k}{e^{n+1}}^2\leq C(j,L) \left(\normSigma{k}{e^n}^2+\normSigma{k}{f}^2 \right).
  \end{align}
\end{proof}

Thanks to these lemmas, we are now ready to prove Equation \eqref{eq:ggoal1}
for $p\in\{1,2,3\}$, and Equation \eqref{eq:ggoal2}.
We begin with Equation \eqref{eq:ggoal1}.

\begin{itemize}
  \item Estimate of $\Re\esp{\indic{\tau(\delta t)>n+1}\langle a_{1,1},A^k e^{n+1/2}\rangle}$.
  We recall that
  \begin{align*}
    a_{1,1}=-i\delta t A_K e^{n+1/2}.
  \end{align*}
  Then $\langle a_{1,1},A^k e^{n+1/2}\rangle$ is purely imaginary,
  and it follows that
  \begin{align}
    \label{eq:estim_a1}
    \Re\esp{\indic{\tau(\delta t)>n+1}\langle a_{1,1},A^k e^{n+1/2}\rangle}=0.
  \end{align}
  \vspace{2mm}

  \item Estimate of $\Re\esp{\indic{\tau(\delta t)>n+1}\langle a_{2,1},A^k e^{n+1/2}\rangle}$.
  We recall that
  \begin{align*}
    a_{2,1}=-i\sqrt{\delta t}\chi^{n+1}B_Ke^{n+1/2}.
  \end{align*}
  Then,
  \begin{align*}
    \Re\langle A^k e^{n+1/2}, a_{2,1}\rangle
    &=-\sqrt{\delta t}\chi^{n+1}\Im\langle A^k e^{n+1/2},B_K e^{n+1/2} \rangle\\
    &=\dfrac{1}{2}\sqrt{\delta t}\chi^{n+1}\Im\langle [A^k,B_K] e^{n+1/2},e^{n+1/2} \rangle.
  \end{align*}
  We split this expression in the following way.
  \begin{align*}
    \Re\langle a_{2,1},A^k e^{n+1/2}\rangle
    \leq&\dfrac{1}{4}\sqrt{\delta t}\chi^{n+1}\Im\langle[A^k,B_K]e^{n+1/2},e^{n+1}-e^n\rangle\\
    &+\dfrac{1}{4}\sqrt{\delta t}\chi^{n+1}\Im\langle[A^k,B_K]e^{n},e^{n+1}-e^n\rangle\\
    &+\dfrac{1}{2}\sqrt{\delta t}\chi^{n+1}\Im\langle[A^k,B_K]e^{n},e^{n}\rangle\\
    =&I+II+III.
  \end{align*}
  Using the following triangular inequality,
  \begin{align}
    \label{eq:estim_a21_1}
    \abs{\Re\esp{\indic{\tau(\delta t)>n+1}\langle a_{2,1},A^k e^{n+1/2}\rangle}}
    \leq \esp{\abs{I}}+\esp{\abs{II}}+\abs{\esp{\indic{\tau(\delta t)>n+1} III}},
  \end{align}
  we estimate now each terms in the right-hand side of Equation \eqref{eq:estim_a21_1}.
  \begin{itemize}
    \item Estimate of $\esp{\abs{I}}$.
    We recall that
    \begin{align*}
      \abs{I}
      =\abs{\dfrac{1}{4}\sqrt{\delta t}\chi^{n+1}\Im\langle[A^k,B_K]e^{n+1/2},e^{n+1}-e^n\rangle}.
    \end{align*}
    Plugging the expression of $e^{n+1}-e^n$ given by \eqref{eq:decomp_e} into this expression,
    and using the triangular inequality gives,
    \begin{align*}
      \abs{I}
      \leq&\abs{\dfrac{1}{4}\delta t^{1/2}\chi^{n+1}\Im\langle[A^k,B_K]e^{n+1/2},a_{1,1}\rangle}\\
      &+\abs{\dfrac{1}{4}\delta t^{1/2}\chi^{n+1}\Im\langle[A^k,B_K]e^{n+1/2},a_{2,1}\rangle}\\
      &+\abs{\dfrac{1}{4}\delta t^{1/2}\chi^{n+1}\Im\langle[A^k,B_K]e^{n+1/2},a_{3,1}\rangle}\\
      &+\abs{\dfrac{1}{4}\delta t^{1/2}\chi^{n+1}\Im\langle[A^k,B_K]e^{n+1/2},f\rangle}\\
      =&i+ii+iii+iv.
    \end{align*}
    By replacing $a_{1,1}$, $a_{2,1}$ and $a_{3,1}$ by their definitions gives,
    \begin{align*}
      \abs{I}
      \leq&\abs{\dfrac{1}{4}\delta t^{3/2}\chi^{n+1}\Re\langle[A^k,B_K]e^{n+1/2},A_K e^{n+1/2}\rangle}\\
      &+\abs{\dfrac{1}{4}\delta t(\chi^{n+1})^2\Re\langle[A^k,B_K]e^{n+1/2},B_K e^{n+1/2}\rangle}\\
      &+\abs{\dfrac{1}{4}\delta t^{3/2}\chi^{n+1}\Re\langle[A^k,B_K]e^{n+1/2},f(\phi^n)-f(\phi(t_n))\rangle}\\
      &+\abs{\dfrac{1}{4}\delta t^{1/2}\chi^{n+1}\Im\langle[A^k,B_K]e^{n+1/2},f\rangle}\\
      =&i+ii+iii+iv.
    \end{align*}
    By use of Assumption \ref{as:1}, we can easily get,
    \begin{align*}
      \esp{i}+\esp{ii}
      \leq C(k,L)\delta t \esp{((\chi^{n+1})^2+\sqrt{\delta t}K\abs{\chi^{n+1}})}\normSigma{k}{e^{n+1/2}}^2.
    \end{align*}
    Since $e^{n+1/2}$ is not $\F_n$-measurable, we now use the almost sure bound
    on $\normSigma{k}{e^{n+1/2}}^2$ given by Lemma \ref{lem:conv_lem5}, followed
    by Cauchy-Schwarz inequality,
    \begin{align*}
      \esp{i}+\esp{ii}
      \leq& C(k)\delta t \esp{((\chi^{n+1})^2+\sqrt{\delta t}K\abs{\chi^{n+1}})
      \left(\normSigma{k}{e^{n}}^2+\normSigma{k}{f}^2\right)}\\
      \leq& C(k)\delta t \esp{(\chi^{n+1})^2+\sqrt{\delta t}K\abs{\chi^{n+1}}}
      \esp{\normSigma{k}{e^{n}}^2}\\
      &+ C(k)\delta t \esp{((\chi^{n+1})^2+\sqrt{\delta t}K\abs{\chi^{n+1}})^2}^{1/2}\esp{\normSigma{k}{f}^4}^{1/2}.
    \end{align*}

    To estimate expression $iii$, we use the fact that $f$ is Lipschitz,
    Cauchy-Schwarz and Young inequality, to get,
    \begin{align*}
      \esp{iii}
      \leq& C(k,L)\delta t^{3/2} \esp{\abs{\chi^{n+1}}\normSigma{k}{e^{n+1/2}}\normSigma{k}{e^n}}\\
      \leq& C(k,L) \delta t^{3/2} \esp{\abs{\chi^{n+1}}^2}\esp{\normSigma{k}{e^n}^2}\\
      &+C(k,L) \delta t^{3/2}\esp{\normSigma{k}{e^{n+1/2}}^2}.
    \end{align*}

    Eventually, the last term $iv$ is estimated by Young inequality followed by
    Cauchy-Schwarz inequality,
    \begin{align*}
      \esp{iv}
      \leq& C(k) \delta t\esp{\normSigma{k}{e^{n+1/2}}^2}+C(k)\esp{(\chi^{n+1})^2}^{1/2}\esp{\normSigma{k}{f}^4}^{1/2}.
    \end{align*}

    By collecting the estimations on $i$, $ii$, $iii$ and $iv$, and using Lemmas
    \ref{lem:conv_lem3} and \ref{lem:conv_lem5},
    it comes the following estimation of $\abs{I}$,
    \begin{align}
      \label{eq:estim_I}
      \esp{\abs{I}}\leq C(k,K^2\delta t^{1/2},L,T)\left(
        \delta t\esp{\normSigma{k}{e^n}^2}
        +\delta t^3 \normSigma{k+8}{\phi_0}^2
      \right).
    \end{align}

    \item Estimate of $\esp{\abs{II}}$.
    We recall that
    \begin{align*}
      \abs{II}
      =\abs{\dfrac{1}{4}\sqrt{\delta t}\chi^{n+1}\Im\langle[A^k,B_K]e^{n},e^{n+1}-e^n\rangle}.
    \end{align*}
    Similarly to the estimation of $\esp{\abs{I}}$, we plug the expression of
    $e^{n+1}-e^n$ given by \eqref{eq:decomp_e} into this expression,
    and using the triangular inequality gives,
    \begin{align*}
      \abs{II}
      \leq&\abs{\dfrac{1}{4}\delta t^{1/2}\chi^{n+1}\Im\langle[A^k,B_K]e^{n},a_{1,1}\rangle}\\
      &+\abs{\dfrac{1}{4}\delta t^{1/2}\chi^{n+1}\Im\langle[A^k,B_K]e^{n},a_{2,1}\rangle}\\
      &+\abs{\dfrac{1}{4}\delta t^{1/2}\chi^{n+1}\Im\langle[A^k,B_K]e^{n},a_{3,1}\rangle}\\
      &+\abs{\dfrac{1}{4}\delta t^{1/2}\chi^{n+1}\Im\langle[A^k,B_K]e^{n},f\rangle}\\
      =&i+ii+iii+iv.
    \end{align*}
    By replacing $a_{1,1}$, $a_{2,1}$ and $a_{3,1}$ by their definitions gives,
    \begin{align*}
      \abs{II}
      \leq&\abs{\dfrac{1}{4}\delta t^{3/2}\chi^{n+1}\Re\langle[A^k,B_K]e^{n},A_K e^{n+1/2}\rangle}\\
      &+\abs{\dfrac{1}{4}\delta t(\chi^{n+1})^2\Re\langle[A^k,B_K]e^{n},B_K e^{n+1/2}\rangle}\\
      &+\abs{\dfrac{1}{4}\delta t^{3/2}\chi^{n+1}\Re\langle[A^k,B_K]e^{n},f(\phi^n)-f(\phi(t_n))\rangle}\\
      &+\abs{\dfrac{1}{4}\delta t^{1/2}\chi^{n+1}\Im\langle[A^k,B_K]e^{n},f\rangle}\\
      =&i+ii+iii+iv.
    \end{align*}
    Terms $i$, $iii$ and $iv$ can be estimated with the same techniques we used
    to estimate expression $I$, and we get,
    \begin{align}
      \label{eq:i-iii-iv}
      \esp{i+iii+iv}\leq C(k,K^2\delta t^{1/2},L,T)\left(
        \delta t\esp{\normSigma{k}{e^n}^2}
        +\delta t^3 \normSigma{k+8}{\phi_0}^2
        \right).
    \end{align}

    To estimate expression $ii$ we use the following technique.
    We split $ii$ in the following way,
    \begin{align*}
      ii
      \leq& \abs{\dfrac{1}{8}\delta t (\chi^{n+1})^2\Re\langle[A^k,B_K]e^{n},B_K (e^{n+1}-e^n)\rangle}\\
      &+\abs{\dfrac{1}{4}\delta t (\chi^{n+1})^2\Re\langle[A^k,B_K]e^{n},B_K e^{n}\rangle}.
    \end{align*}
    The expectation of the second term in the RHS is easily estimated using Assumption \ref{as:1}
    since $e^n$ is independent of $\chi^{n+1}$, and gives,
    \begin{align*}
      \esp{\abs{\dfrac{1}{4}\delta t (\chi^{n+1})^2\Re\langle[A^k,B_K]e^{n},B_K e^{n}\rangle}}
      \leq C(k)\delta t \esp{\normSigma{k}{e^n}^2}.
    \end{align*}
    To estimate the first term in the RHS, we substitute again $e^{n+1}-e^n$
    by its expression given by \eqref{eq:decomp_e}, and use the triangular inequality,
    \begin{align*}
      &\abs{\dfrac{1}{8}\delta t (\chi^{n+1})^2\Im\langle[A^k,B_K]e^{n},B_K (e^{n+1}-e^n)\rangle}\\
      &\qquad\leq\abs{\dfrac{1}{8}\delta t^2 (\chi^{n+1})^2\Im\langle[A^k,B_K]e^{n},B_K A_Ke^{n+1/2}\rangle}\\
      &\qquad+\abs{\dfrac{1}{8}\delta t^{3/2} (\chi^{n+1})^3\Im\langle[A^k,B_K]e^{n},B_K^2 e^{n+1/2}\rangle}\\
      &\qquad+\abs{\dfrac{1}{8}\delta t^{2} (\chi^{n+1})^2\Im\langle[A^k,B_K]e^{n},B_K(f(\phi^n)-f(\phi(t_n)))\rangle}\\
      &\qquad+\abs{\dfrac{1}{8}\delta t^{3/2} (\chi^{n+1})^3\Re\langle[A^k,B_K]e^{n},B_K f\rangle}\\
      &\qquad=\alpha+\beta+\gamma+\delta.
    \end{align*}
    The most difficult term to estimate is $\beta$.
    The estimation of $\alpha$, $\gamma$ and $\delta$ is left to the reader.
    To estimate $\beta$, we use Cauchy-Schwarz inequality
    and Assumption \ref{as:eq:2},
    \begin{align*}
      \esp{\beta}
      \leq& C(k)\delta t^{3/2}\esp{\abs{\chi^{n+1}}^3\normSigma{k}{e^n}\normSigma{k}{B_K^2 e^{n+1/2}}}\\
      \leq& C(k)\delta t(\delta t^{1/2}K^2)\esp{\abs{\chi^{n+1}}^3\normSigma{k}{e^n}\normSigma{k}{e^{n+1/2}}}.
    \end{align*}
    Now using Young's inequality and Lemma \ref{lem:conv_lem5},
    followed by Lemma \ref{lem:conv_lem3},
    \begin{align*}
      \esp{\beta}
      \leq& C(k)\delta t (\delta t^{1/2}K^2)\left(
      \esp{\abs{\chi^{n+1}}^6 \normSigma{k}{e^n}^2} + \esp{\normSigma{k}{f}^2}\right)\\
      \leq& C(k)\delta t (\delta t^{1/2}K^2)\left(
      \esp{\abs{\chi^{n+1}}^6}\esp{\normSigma{k}{e^n}^2} + C(L,k,T)\delta t^3\normSigma{k+8}{\phi_0}^2 \right).
    \end{align*}
    This last estimation can be written as,
    \begin{align}
      \label{eq:ii}
      \esp{\beta}
      \leq& C(k,\delta t^{1/2}K^2,L,T)\left(
      \delta t \esp{\normSigma{k}{e^n}^2} + \delta t^3 \normSigma{k+8}{\phi_0}^2
      \right),
    \end{align}
    and combining \eqref{eq:i-iii-iv} and \eqref{eq:ii} leads to the following
    estimation for $\esp{\abs{II}}$,
    \begin{align}
      \label{eq:estim_II}
      \esp{\abs{II}}
      \leq& C(k,\delta t^{1/2}K^2,L,T)\left(
      \delta t \esp{\normSigma{k}{e^n}^2} + \delta t^3 \normSigma{k+8}{\phi_0}^2
      \right).
    \end{align}

    \item Estimation of $\abs{\esp{\indic{\tau(\delta t)>n+1} III}}$.
    It follows from the definition of $\tau(\delta t)$ (see \eqref{eq:tau_lin}) that,
    \begin{align*}
      \indic{\tau(\delta t)>n+1}=\indic{\tau(\delta t)>n}\indic{\abs{\chi^{n+1}}< C_0\delta t^{-1/2}}.
    \end{align*}
    Using now the fact that $\chi^{n+1}\indic{\abs{\chi^{n+1}}< C_0\delta t^{-1/2}}$ is independent of $e^n$ and $\indic{\tau(\delta t)>n}$, it follows
    \begin{align*}
      \esp{\indic{\tau(\delta t)>n+1}III}=&\esp{\chi^{n+1}\indic{\abs{\chi^{n+1}}< C_0\delta t^{-1/2}}}\\
      &\times \esp{\indic{\tau(\delta t)>n}\dfrac{1}{2}\sqrt{\delta t}\Im\langle[A^k,B_K]e^{n},e^{n}\rangle}.
    \end{align*}
    Since
    \begin{align*}
      \esp{\chi^{n+1}\indic{\abs{\chi^{n+1}}< C_0\delta t^{-1/2}}}=0,
    \end{align*}
    it follows that,
    \begin{align}
      \label{eq:estim_III}
      \esp{\indic{\tau(\delta t)>n+1}III}=0.
    \end{align}
  \end{itemize}
  Eventually, combining \eqref{eq:estim_I}, \eqref{eq:estim_II} and
  \eqref{eq:estim_III} gives,
  \begin{align}
    \label{eq:estim_a2}
    \abs{\Re\esp{\indic{\tau(\delta t)>n+1}\langle a_{2,1},A^k e^{n+1/2}\rangle}}
    \leq C(k,\delta t^{1/2}K^2,L,T)\left(
    \delta t \esp{\normSigma{k}{e^n}^2} + \delta t^3 \normSigma{k+8}{\phi_0}^2
    \right).
  \end{align}
  \vspace{2mm}

  \item Estimate of $\Re\esp{\indic{\tau(\delta t)>n+1}\langle a_{3,1},A^k e^{n+1/2}\rangle}$.
  We recall that
  \begin{align*}
    a_{3,1}=-i\delta t P_K\left( g(\phi^{n+1},\phi^n)-g(\phi(t_{n+1}),\phi(t_n)) \right).
  \end{align*}
  Using Cauchy-Schwarz inequality and the fact that $g$ is Lipschitz for each
  variable in $\Sigma^k(\R^d)$, it easily comes
  \begin{align*}
    \abs{\Re\esp{\indic{\tau(\delta t)>n+1}\langle a_{3,1},A^k e^{n+1/2}\rangle}}
    \leq C(k,L)\delta t \left(\esp{\normSigma{k}{e^n}^2}+\esp{\normSigma{k}{e^{n+1}}^2} \right).
  \end{align*}
  Using now Lemma \ref{lem:conv_lem5}, it comes,
  \begin{align*}
    \abs{\Re\esp{\indic{\tau(\delta t)>n+1}\langle a_{3,1},A^k e^{n+1/2}\rangle}}
    \leq C(k,L)\delta t \left(\esp{\normSigma{k}{e^n}^2}+\esp{\normSigma{k}{f}^2} \right).
  \end{align*}
  Eventually using Lemma \ref{lem:conv_lem3}, it follows that,
  \begin{align}
    \label{eq:estim_a3}
    \abs{\Re\esp{\indic{\tau(\delta t)>n+1}\langle a_{3,1},A^k e^{n+1/2}\rangle}}
    \leq C(k,L,T)\left(
    \delta t \esp{\normSigma{k}{e^n}^2} + \delta t^3 \normSigma{k+8}{\phi_0}^2
    \right).
  \end{align}
  \end{itemize}
  \vspace{2mm}

  We recall that, at this stage of the proof, Equations \eqref{eq:estim_a1},
  \eqref{eq:estim_a2} and \eqref{eq:estim_a3} prove Equation \eqref{eq:ggoal1}
  for $p\in\{1,2,3\}$.
  It remains to prove Equation \eqref{eq:ggoal2}, which is done now:
  \begin{itemize}
  \item Estimate of $\Re\esp{\indic{\tau(\delta t)>n+1}\langle f,A^k e^{n+1/2}\rangle}$
  This estimation uses extensively the Lemma \ref{lem:conv_lem4}.
  We introduce $f_1$ and $f_2$, given this Lemma, and the triangular inequality
  to get,
  \begin{align*}
    \abs{\Re\esp{\indic{\tau(\delta t)>n+1}\langle f,A^k e^{n+1/2}\rangle}}
    \leq
    &\abs{\Re\esp{\indic{\tau(\delta t)>n+1}\langle f_1,A^k e^{n+1/2}\rangle}}\\
    &+\abs{\Re\esp{\indic{\tau(\delta t)>n+1}\langle f_2,A^k e^{n+1/2}\rangle}}\\
    =& (i)+(ii).
  \end{align*}
  We estimate now $(i)$ and $(ii$).
  \begin{itemize}
    \item Estimation of $(i)$.
    By Cauchy-Schwarz and Young inequalities,
    \begin{align*}
      (i)
      \leq C \left (
        \dfrac{\delta t}{2} \esp{\normSigma{k}{e^{n+1/2}}^2}
        + \dfrac{1}{2\delta t} \esp{\normSigma{k}{f_1}^2}
        \right),
    \end{align*}
    and by Lemma \ref{lem:conv_lem4},
    \begin{align}
      \label{estim_fi}
      (i)
      \leq C(L,j,T) \left(\dfrac{\delta t}{2} \esp{\normSigma{k}{e^{n+1/2}}^2}
      + \delta t^3 \normSigma{k+8}{\phi_0}^2 \right).
    \end{align}

    \item Estimation of $(ii)$.
    By triangular inequality, we split $(ii)$ in the following way,
    \begin{align*}
      (ii)\leq
      &\dfrac{1}{2}\esp{\abs{\langle A^2 f_2,A^{k-2} (e^{n+1}-e^n)\rangle}}\\
      &+\abs{\esp{\langle f_2,A^ke^{n+1/2} \rangle}}\\
      &+\abs{\esp{\indic{\tau(\delta t)\leq n+1}\langle f_2,A^ke^{n+1/2} \rangle}}\\
      \leq& (a)+(b)+(c).
    \end{align*}
    Thanks to Lemma \ref{lem:conv_lem4} we get
    \begin{align}
      \label{estim_fb}
      (b)=0,
    \end{align}
    and using in addition Young's inequality,
    \begin{align}
      \label{estim_fa}
      (a)\leq C(L,j,T)\left(\esp{\normSigma{k}{e^n}^2}+\delta t^3\normSigma{k+12}{\phi_0}^2\right).
    \end{align}
    We use Young's and Cauchy-Schwarz inequalities to estimate $(c)$ ,
    \begin{align*}
      (c)
      &\leq \delta t\esp{\normSigma{k}{e^{n+1}}^2}+\dfrac{1}{\delta t}\esp{\indic{\tau(\delta t)\leq n+1}\normSigma{k}{f_2}^2}\\
      &\leq \delta t\esp{\normSigma{k}{e^{n+1}}^2}+\dfrac{1}{\delta t}\prob{\tau(\delta t)\leq n+1}^{1/2}\esp{\normSigma{k}{f_2}^4}^{1/2}.
    \end{align*}
    Using again Lemma \ref{lem:conv_lem4} and the fact that there exists $C(T)>0$ such
    that
    \begin{align*}
      \prob{\tau(\delta t)\leq n+1}^{1/2}\leq C(T)\delta t
    \end{align*}
    holds uniformly in $\delta t$,
    \begin{align}
      \label{estim_fc}
      (c)
      &\leq \delta t\esp{\normSigma{k}{e^{n+1}}^2}+C(L,k,T)\delta t^3\normSigma{k+8}{\phi_0}.
    \end{align}
  \end{itemize}
  Then, by collecting \eqref{estim_fi}, \eqref{estim_fb}, \eqref{estim_fa}, and \eqref{estim_fc},
  we get that
  \begin{align}
    \label{eq:proved_ggoal2}
    \Re\esp{\indic{\tau(\delta t)>n+1}\langle f,A^k e^{n+1/2}\rangle}
    \leq C(k,L,T)\left(
    \delta t \esp{\normSigma{k}{e^n}^2} + \delta t^3 \normSigma{k+8}{\phi_0}^2
    \right),
  \end{align}
  which proves Equation \eqref{eq:ggoal2}.
\end{itemize}

Since Equations \eqref{eq:ggoal1} and \eqref{eq:ggoal2} have been proved,
it immediately follows Equation \eqref{eq:goal_th1}:
\begin{align*}
  \esp{\normSigma{k}{e^{n+1}}^2}
  \leq (1+\delta t C(T,k,L,K_0))\esp{\normSigma{k}{e^{n}}^2}
  +\delta t^3 C(T,k,L,K_0)\normSigma{k+12}{\phi_0}^2,
\end{align*}
and the discrete Gronwall's Lemma gives that,
\begin{align*}
  \esp{\normSigma{k}{e^{n+1}}^2}\leq e^{C(T,k,L,K_0)n\delta t}\delta t^2 C(T,k,L,K_0)\normSigma{k+12}{\phi_0}^2,
\end{align*}
and
\begin{align*}
  \sup_{n\leq N} \esp{\normSigma{k}{e^{n}}^2}\leq C(T,k,L,K_0) \delta t^2 C(T,k,L,K_0)\normSigma{k+12}{\phi_0}^2.
\end{align*}
\end{proof}
\vspace{3mm}

\begin{proof}[Proof of Proposition \ref{prop:conv_ms2}]
  We begin with recalling that since the nonlinearity is truncated to be Lipschitz
  in $\Sigma^k(\R^d)$, it follows that for all $L>0$ and $T>0$, there exists $C(k,L,T)>0$ such that
  for all $\phi_0\in\Sigma^{k+12}(\R^d)$,
  \begin{align}
    \label{eq:maj_trunc}
    \sup_{t\leq T}\esp{\normSigma{k+12}{\phi_{L}(t)}^2}
    +\sup_{t\leq T}\esp{\normSigma{k+12}{\phi_{K(\delta t),L}(t)}^2}
    \leq C(k,L,T) \normSigma{k+12}{\phi_0}^2,
  \end{align}
  where $\phi_L$ is the solution of Equation \eqref{eq:GP_trunc}
  and $\phi_{K(\delta t),L}$ is the solution of Equation \eqref{eq:GP_modif}.
  \vspace{2mm}

  We define the $\F_t$-measurable process $\Phi_{K(\delta t),L}$ for all $t\geq 0$ by
  $\Phi_{K(\delta t),L}(t):=\phi_L(t)-\phi_{K(\delta t),L}(t)$.
  With this notation, we aim to show that
  \begin{align}
    \label{eq:goalprop4}
    \sup_{t\leq T}\esp{\normSigma{k}{\Phi_{K(\delta t),L}(t)}^2} \leq C(T,k,L,K_1) \delta t^2.
  \end{align}
  The proof is done by using It\^{o}'s lemma, followed by Gronwall's inequality.
  More precisely, we apply It\^{o}'s lemma to $\phi_L$ and $\phi_{K(\delta t),L}$
  for the functional $(\phi_1,\phi_2)\mapsto \Re\langle A^k(\phi_1-\phi_2),(\phi_1-\phi_2)\rangle=\normSigma{k}{\phi_1-\phi_2}^2$.
  To make clear how the computation is done, we begin with recalling
  It\^{o}'s formulations of \eqref{eq:GP_trunc} and \eqref{eq:GP_modif}
  where we denote $K$ instead of $K(\delta t)$ to simplify notations,
  \begin{align*}
      d\phi_{K,L} =
      -i A_K\phi_{K,L}dt
      -i B_K\phi_{K,L} dW_t
      -\dfrac{1}{2} B_K^2\phi_{K,L} dt
      -i \lambda P_K f^k_L(\phi_{K,L}) dt,
  \end{align*}
  and
  \begin{align*}
    d\phi_L =&
    -iA\phi_L dt
    -i\abs{x}^2\phi_L dW_t
    -\dfrac{1}{2}\abs{x}^4\phi_L dt
    -i\lambda f^k_L(\phi_L)dt\\
    =& -iA_K \phi_L dt -iA(Id-P_K)\phi_L dt
    -i\abs{x}^2\phi_L dW_t\\
    &-\dfrac{1}{2}B_K^2\phi_L dt -\dfrac{1}{2}(\abs{x}^4 - B_K^2)\phi_L dt
    -i\lambda P_K f^k_L(\phi_L)dt -i\lambda (Id-P_K) f^k_L(\phi_L)dt.
  \end{align*}
  Then, the multidimensional It\^{o}'s lemma gives,
  \begin{equation}
    \label{eq:bigIto}
    \begin{aligned}
    \normSigma{k}{\Phi_{K,L}(t)}^2
    =&\normSigma{k}{\Phi_{K,L}(0)}^2\\
    &+ 2\int_0^t \Re \langle A^k \Phi_{K,L}(s),-iA_K \Phi_{K,L}(s) \rangle ds\\
    &+ 2\int_0^t \Re \langle A^k \Phi_{K,L}(s),-iA (Id-P_K) \phi_L(s) \rangle ds\\
    &+ 2\int_0^t \Re \langle A^k \Phi_{K,L}(s),-i (\abs{x}^2 \phi_L(s)-B_K\phi_{K,L}(s)) \rangle dW_s\\
    &+ 2\int_0^t \Re \langle A^k \Phi_{K,L}(s),- \dfrac{1}{2} B_K^2 \Phi_{K,L}(s) \rangle ds\\
    &+ 2\int_0^t \Re \langle A^k \Phi_{K,L}(s),- \dfrac{1}{2}(\abs{x}^4 - B_K^2) \phi_L(s) \rangle ds\\
    &+ 2 \int_0^t \Re \langle A^k \Phi_{K,L}(s),-i\lambda P_K (f_L^k(\phi_L(s))-f_L^k(\phi_{K,L}(s))) \rangle ds\\
    &+ 2 \int_0^t \Re \langle A^k \Phi_{K,L}(s),-i\lambda (Id-P_K) f_L^k(\phi_L(s)) \rangle ds\\
    &+ \int_0^t \Re \langle A^k (-i\abs{x}^2\phi_L(s)),-i\abs{x}^2\phi_L(s) \rangle ds\\
    &+ \int_0^t \Re \langle A^k (-iB_K \phi_{K,L}(s)),-iB_K \phi_{K,L}(s) \rangle ds\\
    &-2 \int_0^t \Re \langle A^k (-i\abs{x}^2 \phi_{L}(s)),-iB_K \phi_{K,L}(s) \rangle ds.
    \end{aligned}
  \end{equation}
  The last three terms in the right-hand side are the It\^{o} correction.
  The last line takes into accounts the cross derivatives of the functional.
  We write this correction in another way, that will enable us to use Cauchy-Schwarz
  inequality, by noticing that,
  \begin{align*}
    &\Re \langle A^k \abs{x}^2 \phi_L(s),\abs{x}^2\phi_L(s)\rangle
    +\Re \langle A^k B_K\phi_{K,L}(s),B_K\phi_{K,L}(s)\rangle
    -2\Re \langle A^k \abs{x}^2 \phi_L(s) ,B_K \phi_{K,L}(s) \rangle\\
    &\qquad=\Re \langle A^k B_K \Phi_{K,L}(s),B_K \phi_{K,L}(s)\rangle
    +\Re \langle A^k (\abs{x}^2-B_K)\phi_L(s),(\abs{x}^2-B_K)\phi_L(s)\rangle\\
    &\qquad+2\Re \langle A^k B_K \Phi_{K,L}(s),(\abs{x}^2-B_K)\phi_L(s) \rangle.
  \end{align*}
  Then, Equation \eqref{eq:bigIto} can be written as,
  \begin{equation}
    \label{eq:bigIto2}
    \begin{aligned}
    \normSigma{k}{\Phi_{K,L}(t)}^2
    =&\normSigma{k}{\Phi_{K,L}(0)}^2\\
    &+ 2\int_0^t \Re \langle A^k \Phi_{K,L}(s),-iA_K \Phi_{K,L}(s) \rangle ds\\
    &+ 2\int_0^t \Re \langle A^k \Phi_{K,L}(s),-iA (Id-P_K) \phi_L(s) \rangle ds\\
    &+ 2\int_0^t \Re \langle A^k \Phi_{K,L}(s),-i (\abs{x}^2 \phi_L(s)-B_K\phi_{K,L}(s)) \rangle dW_s\\
    &+ 2\int_0^t \Re \langle A^k \Phi_{K,L}(s),- \dfrac{1}{2} B_K^2 \Phi_{K,L}(s) \rangle ds\\
    &+ 2\int_0^t \Re \langle A^k \Phi_{K,L}(s),- \dfrac{1}{2}(\abs{x}^4 - B_K^2) \phi_L(s) \rangle ds\\
    &+ 2 \int_0^t \Re \langle A^k \Phi_{K,L}(s),-i\lambda P_K (f_L^k(\phi_L(s))-f_L^k(\phi_{K,L}(s))) \rangle ds\\
    &+ 2 \int_0^t \Re \langle A^k \Phi_{K,L}(s),-i\lambda (Id-P_K) f_L^k(\phi_L(s)) \rangle ds\\
    &+\int_0^t\Re \langle A^k B_K\Phi_{K,L}(s),B_K\Phi_{K,L}(s) \rangle ds\\
    &+\int_0^t\Re \langle A^k (\abs{x}^2-B_K)\phi_L(s),(\abs{x}^2-B_K)\phi_L(s)\rangle ds\\
    &+2\int_0^t\Re \langle A^k B_K \Phi_{K,L}(s),(\abs{x}^2-B_K)\phi_L(s) \rangle ds.
    \end{aligned}
  \end{equation}
  We proceed now to estimate all the terms that appear in the right-hand side
  of Equation \eqref{eq:bigIto2},
  \begin{itemize}
    \item the second term is equal to zero since
    $\Re \langle A^k \Phi_{K,L}(s),-iA_K \Phi_{K,L}(s) \rangle=0$.
    \item Using Cauchy-Schwarz and Young inequalities, and Equation \eqref{eq:diff_trunc},
    it follows that the third term can be bounded in the following way,
    \begin{align*}
      \Re \langle A^k \Phi_{K,L}(s),-iA (Id-P_K) \phi_L(s) \rangle
      \leq C\left( \normSigma{k}{\Phi_{K,L}(s)}^2 + K^{-8}\normSigma{k+10}{\phi_L(s)}^2 \right).
    \end{align*}
    \item The expectation of the stochastic integral vanishes.
    \item Collecting the fifth and the ninth terms,
    \begin{align*}
      &2\Re \langle A^k \Phi_{K,L}(s),- \dfrac{1}{2} B_K^2 \Phi_{K,L}(s) \rangle
      +\Re \langle A^k B_K\Phi_{K,L}(s),B_K\Phi_{K,L}(s) \rangle\\
      &\qquad= \Re \langle [A^k, B_K] \Phi_{K,L}(s),B_K\Phi_{K,L}(s) \rangle,
    \end{align*}
    and by \eqref{as:eq:4},
    \begin{align*}
      &2\Re \langle A^k \Phi_{K,L}(s),- \dfrac{1}{2} B_K^2 \Phi_{K,L}(s) \rangle
      +\Re \langle A^k B_K\Phi_{K,L}(s),B_K\Phi_{K,L}(s) \rangle\\
      &\qquad\leq C \normSigma{k}{\Phi_{K,L}(s)}^2.
    \end{align*}
    \item Using the fact that $\abs{x}^4-B_K^2=\abs{x}^2(\abs{x}^2-B_K)+(\abs{x}^2-B_K)B_K$,
    the sixth term can be bounded by,
    \begin{align*}
      \Re \langle A^k \Phi_{K,L}(s),- \dfrac{1}{2}(\abs{x}^4 - B_K^2) \phi_L(s) \rangle
      \leq C \left( \normSigma{k}{\Phi_{K,L}(s)}^2 + K^{-8} \normSigma{k+12}{\phi_L(s)}^2\right),
    \end{align*}
    similarly to the third term.
    \item The estimation of the seventh term uses the truncation of nonlinearity,
    \begin{align*}
      \Re \langle A^k \Phi_{K,L}(s),-i\lambda P_K (f_L^k(\phi_L(s))-f_L^k(\phi_{K,L}(s))) \rangle
      \leq C(L,k) \normSigma{k}{\Phi_{K,L}(s)}^2.
    \end{align*}
    The estimation of the eighth term uses in addition Cauchy-Schwarz and Young inequality,
    and \eqref{eq:diff_trunc},
    \begin{align*}
      \Re \langle A^k \Phi_{K,L}(s),-i\lambda (Id-P_K) f_L^k(\phi_L(s)) \rangle
      \leq C(L,k) \left(\normSigma{k}{\Phi_{K,L}(s)}^2 + K^{-8}\normSigma{k+8}{\phi_L(s)}^2  \right).
    \end{align*}
    \item The tenth term is estimated by Assumption \ref{as:convergence},
    \begin{align*}
        \Re \langle A^k (\abs{x}^2-B_K)\phi_L(s),(\abs{x}^2-B_K)\phi_L(s)\rangle
        \leq C K^{-8} \normSigma{k+10}{\phi_L(s)}^2.
    \end{align*}
    \item The last term is estimated by replacing $A^kB_K$ by $B_K A^k+[A^k,B_K]$,
    and using the triangular inequality and Assumption \ref{as:convergence},
    \begin{align*}
      \Re \langle A^k B_K \Phi_{K,L}(s),(\abs{x}^2-B_K)\phi_L(s)
      \leq C\left( \normSigma{k}{\Phi_{K,L}(s)}^2 + K^{-8} \normSigma{k+12}{\phi_L(s)}^2 \right).
    \end{align*}
  \end{itemize}
  Collecting all these estimations, it follows that, for all $t\leq T$
  \begin{align*}
    \esp{\normSigma{k}{\Phi_{K(\delta t),L}(t)}^2}
    \leq&\normSigma{k}{\Phi_{K(\delta t),L}(0)}^2
    +C(k,L)\int_0^t\esp{\normSigma{k}{\Phi_{K(\delta t),L}(s)}^2}ds\\
    &+C(k,L) K^{-8}\int_0^t\esp{\normSigma{k+12}{\phi_{L}(s)}^2}ds.
  \end{align*}
  Using now \eqref{eq:maj_trunc} and the fact that
  $\normSigma{k}{\Phi_{K(\delta t),L}(0)}^2\leq C K^{-8}\normSigma{k+8}{\phi_0}^2$,
  it comes that for all $t\leq T$,
  \begin{align*}
    \esp{\normSigma{k}{\Phi_{K(\delta t),L}(t)}^2}
    \leq C(k,L,T) K^{-8} \normSigma{k+12}{\phi_0}^2
    + C(k,L) \int_0^t\esp{\normSigma{k}{\Phi_{K(\delta t),L}(s)}^2}ds,
  \end{align*}
  and the application of Gronwall's inequality
  with the fact that we choose $K^{-8}\leq K_1^{-4}\delta t^2$
  enables to conclude the proof.
\end{proof}

\section{Numerical experimentations}
\label{sec:num}

This part is devoted to numerical experimentations in space dimension one
with the scheme defined above.
We aim to show that we obtain good results using a spectral discretization
based on Hermite functions.
We also compare the Crank-Nicolson time discretization against splitting methods.
Four different schemes are compared in the following.
We begin with explaining how to implement the scheme defined by \eqref{eq:scheme},
and then we present the other schemes, which are some adaptations of classical schemes
for the deterministic Gross-Pitaevskii equation.
\vspace{3mm}

Two of them rely on the computation of a discrete Hermite transform at each time-step.
This transform decomposes a function $f\in L^2(\R)$ on the basis of $L^2(\R)$
composed by Hermite functions, which are the eigenvectors of operator $-\Delta+\abs{x}^2$.
The discrete transform on the $K$ first modes enables to exactly compute the
components of any function $f\in\Sigma_K(\R)$ in this basis.
It can be computed using the Gauss–Hermite quadrature.
Obviously, for general $f\in L^2(\R)$, the discrete Hermite transform does not
exactly compute these projections.
The error between the values of the components of $f$ on the $K$ first Hermite functions,
and the values given by the discrete transform can be estimated with respect to
the regularity of those functions (see \cite{Mastroianni1994}).
A naive implementation (by a classical matrix-vector product) of the discrete Hermite transform algorithm (in 1D)
using the Gauss-Hermite quadrature
on the $K$ first Hermite functions
leads to a computational cost of order $K^2$.
Such an implementation can be generalized to the $d$-dimensional case with
computational costs of order $K^{d+1}$.
Nevertheless, there exists efficient implementations with computational
costs of order $K^d\log^2(K)$
(see \cite{Leibon:2008:FHT,Potts:1998:FAD,Driscoll:1997:FDP}).
In the following we will say that a computation is \emph{almost exact} when the
only approximations come from the fact that we use the discrete Hermite transform
to compute the projections on the $K$ first Hermite functions
for a general function $f\in L^2(\R^d)$, which does not necessarily belong
to $\Sigma_K(\R^d)$.
The same denomination will be used for the Discrete Fourier Transform.
\vspace{3mm}

With the explicit example of operator $B_K$ we gave (with Equations \eqref{eq:def_BK_1},
\eqref{eq:def_BK_2}, \eqref{eq:def_BK_3} and \eqref{eq:def_BK_4}),
the solutions of Equation \eqref{eq:scheme}
can be computed \emph{almost exactly} thanks to a fixed-point iterative algorithm.
The convergence of such an algorithm is ensured if $\delta t$ is small enough
with respect to $L$, the level of truncation of the nonlinearity, as it is shown
in the proof of Proposition \ref{prop:exist_scheme}.
Nevertheless, the smooth spectral cut-off of operator $\abs{x}^2$ doesn't seem
to have a significant impact in practical cases.
Thus, in practice one can simply take the symmetric operator $P_K\abs{x}^2P_K$ for $B_K$.
\vspace{3mm}

Moreover, in practical cases, the truncations of the non-linearity and the normal deviates
$(\chi^n)_{n\in\N^*}$ do not seem to be significant and can be omitted in the implementation.
We denote by $(\phi_K^n)$ the numerical approximation in this case such that
for all $n\in\N$, $\phi_K^n\in M_{K,1}(\R)$ (column vector of size $K$) and
its components are the components in the Hermite basis of $L^2(\R)$.
Then, Equation \eqref{eq:scheme} is given by
\begin{align}
  \label{eq:schemeimpl}
  D(\delta t,\chi^{n+1})\phi_K^{n+1}=C(\delta t,\chi^{n+1})\phi_K^n -i\delta t P_K g(\phi_K^{n+1},\phi_K^n),
\end{align}
where $C$ and $D$ are the operators defined on $\Sigma_K(\R)$ by
\begin{align}
  \label{eq:defC}
  C(\delta t,\chi^{n+1})=\left(Id-i\dfrac{\delta t}{2} A-i\dfrac{\sqrt{\delta t}\chi^{n+1}}{2}P_K\abs{x}^2\right),
\end{align}
and,
\begin{align}
  \label{eq:defD}
  D(\delta t,\chi^{n+1})=\left(Id+i\dfrac{\delta t}{2} A+i\dfrac{\sqrt{\delta t}\chi^{n+1}}{2}P_K\abs{x}^2\right).
\end{align}
We recall that, by denoting $(e_k)_{k\in\N}$ the orthonormal basis of $L^2(\R)$ composed by Hermite functions, we have the recursive relation,
\begin{align*}
  \forall x\in\R,\qquad xe_k(x)=\sqrt{\dfrac{k}{2}}e_{k-1}(x)+\sqrt{\dfrac{k+1}{2}}e_{k+1}(x).
\end{align*}
Thus, in 1D the matrices of these two operators are pentadiagonal in the Hermite basis.
The computation of the Hermite modes of $P_K g(\phi_K^{n+1},\phi_K^n)$ can be done almost exactly
using the discrete Hermite transform.
Indeed, knowing $\phi_K^n$ and $\phi_K^{n+1}$ one can compute $g(\phi_K^{n+1},\phi_K^n)(x)$
exactly for all $x\in\R$.
Thus one can use a discrete Hermite transform to compute \emph{almost exactly}
the components of $P_K g(\phi_K^{n+1},\phi_K^n)$ in the Hermite basis.
\vspace{3mm}

We present now three other schemes that we use to compare with the one presented above.
\paragraph{\textbf{Splitting scheme with spectral-Fourier discretization}}
For this scheme, the time discretization is based on a splitting method, and the space
discretization is based on a spectral discretization on the Fourier modes.
In other words, we look for a solution defined on $[-L_x,L_x]$, with periodic boundary conditions,
supported by the $K$ first Fourier modes.
A similar scheme has been proposed, for the deterministic case in \cite{Antoine201595,Bao2003318}
for solving Equation \eqref{eq:GP} (without the Stratonovich noise).
We generalize this scheme to our stochastic setting.
The idea is to solve the following equation,
\begin{equation}
  \label{eq:split11}
  \left\{
    \begin{aligned}
    &d\phi= - i\Delta \phi dt,\\
    &\phi(t_n)=\phi_K^n,
    \end{aligned}
  \right.
\end{equation}
and we set $\phi_K^{n+1/2}=\phi(t_{n+1})$.
Then we solve
\begin{equation}
  \label{eq:split12}
  \left\{
    \begin{aligned}
      &d\psi=-i\abs{x}^2 \psi dt -i\abs{x}^2 \psi\circ dW_t-i g \abs{\psi}^2\psi dt,\\
      &\psi(t_n)=\phi_K^{n+1/2},
    \end{aligned}
  \right.
\end{equation}
and we set $\phi_K^{n+1}=P_K \psi(t_{n+1})$.
Here $P_K$ denotes the projection onto the $K$ first Fourier modes.
The first step is computed exactly in the Fourier space.
To compute the other one, we use a discrete Fourier transform.
An inverse discrete Fourier transform enables to compute the value of $\phi_K^{n+1/2}$
on a uniform grid of $[-L_x,L_x]$ of $(K+1)$ points.
Then, the value of $\psi(t_{n+1})$ on those points $x$ is explicitly given by
\begin{align*}
  \psi(t_{n+1})(x)=e^{-i\abs{x}^2(\delta t+\sqrt{\delta t}\chi^{n+1})-ig\abs{\phi_K^{n+1/2}}^2 \delta t}\phi_K^{n+1/2}(x).
\end{align*}
This computation relies on the fact that $t\mapsto \abs{\psi(t)}^2$ is constant on $[t_n,t_{n+1}]$.
Then, the computation of the first Fourier modes of $\phi_K^{n+1}$ is done \emph{almost exactly}
by a discrete Fourier transform.

\paragraph{\textbf{Splitting scheme with spectral-Hermite discretization}}
This scheme is similar to the previous one, but we chose this time a spectral
discretization based on the first Hermite functions.
Equations \eqref{eq:split11} and \eqref{eq:split12} are replaced by
\begin{equation}
  \label{eq:split21}
  \left\{
    \begin{aligned}
    &d\phi= - i A \phi dt,\\
    &\phi(t_n)=\phi_K^n,
    \end{aligned}
  \right.
\end{equation}
and
\begin{equation}
  \label{eq:split22}
  \left\{
    \begin{aligned}
      &d\psi=-i\abs{x}^2 \psi \circ dW_t - i g \abs{\psi}^2\psi \circ dt,\\
      &\psi(t_n)=\phi_K^{n+1/2},
    \end{aligned}
  \right.
\end{equation}
and $P_K$ denotes now the projection onto the first Hermite functions.
Similarly, both equations \eqref{eq:split21} and \eqref{eq:split22} are solved
exactly, and the projection of $\psi(t_{n+1})$ onto the $K$ first Hermite functions
is done by a discrete Hermite Transform, and thus is not exact.

\paragraph{\textbf{Crank-Nicolson scheme with finite differences discretization}}
In this case, we suppose that the solution almost vanishes outside the space interval
$[-L_x,L_x]$, such that we can set homogeneous Dirichlet conditions on the boundaries
of this domain.
The space discretization of the Laplacian is classical, and the nonlinearity can
be truncated to ensure well-posedness of the scheme.
We discretize $[-L_x,L_x]$ with a regular grid of $2K+1$ points.
We set $\delta x=L_x/K$.
We denote by $\phi^n_k$ an approximation of
$\phi(n\delta t, k\delta x)$ for $k\in\{-K,-K+1,\ldots,K\}$, and
by $\phi^{n+1/2}_k$ by $\phi^{n+1/2}_k=\frac{\phi^{n+1}_k+\phi^n_k}{2}$.
We define the scheme by induction for $n\in\N$ and $k\in\{-K+1,\ldots,K-1\}$ by
setting $\phi^0_k=\phi(0,k\delta x)$ and
\begin{equation}
  \begin{aligned}
    \label{eq:scheme_CN_FD}
    \phi^{n+1}_k&=\phi^{n}_k
    -i \dfrac{\delta t}{\delta x^2}( \phi^{n+1/2}_{k+1}-2\phi^{n+1/2}_k+\phi^{n+1/2}_{k-1} )
    -i (\delta t+\sqrt{\delta t}\chi^{n+1}) \abs{k\delta x}^2 \phi^{n+1/2}_k\\
    &-i \delta t g(\phi^{n+1}_k,\phi^n_k).
  \end{aligned}
\end{equation}
\vspace{3mm}


In the following, we compare the convergence properties for these four schemes.
We begin with estimating the pathwise speed of convergence for the space discretization.
To do that, we choose a fine enough time-step to neglect the effect of the time discretization,
and a geometrically increasing sequence of degrees of freedom for the space discretization.
Then we compare the error between the solutions computed for two successive levels
of space discretization, for one trajectory.
More precisely, we choose $T=10$, $\delta t=10\cdot 2^{-20}$.
For the Fourier and finite differences approximations, we solve the equation
on the domain $[-10,10]$ and we set respectively periodic boundary conditions
and Dirichlet homogeneous.
Moreover, we use a coefficient $\alpha=0.3$ in front of the noise, to decrease its effects.
It enables to reduce significantly the error coming from the time discretization.
The result is given in Figure \ref{fig:convK1}
\begin{figure}[!ht]
\centering
\begin{minipage}{.5\textwidth}
  \centering
  \includegraphics[width=\linewidth]{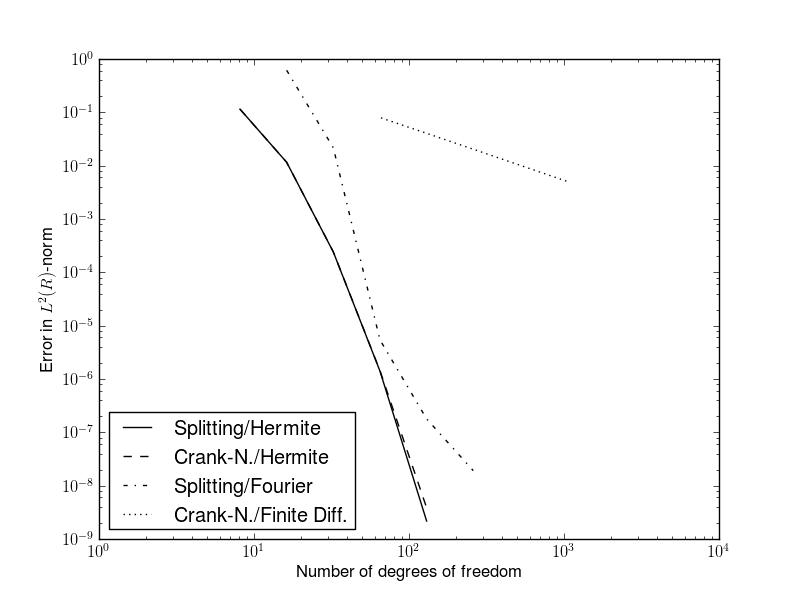}
  \captionof{figure}{Pathwise speed of convergence for each discretization toward their respective limit.}
  \label{fig:convK1}
\end{minipage}%
\begin{minipage}{.5\textwidth}
  \centering
  \includegraphics[width=\linewidth]{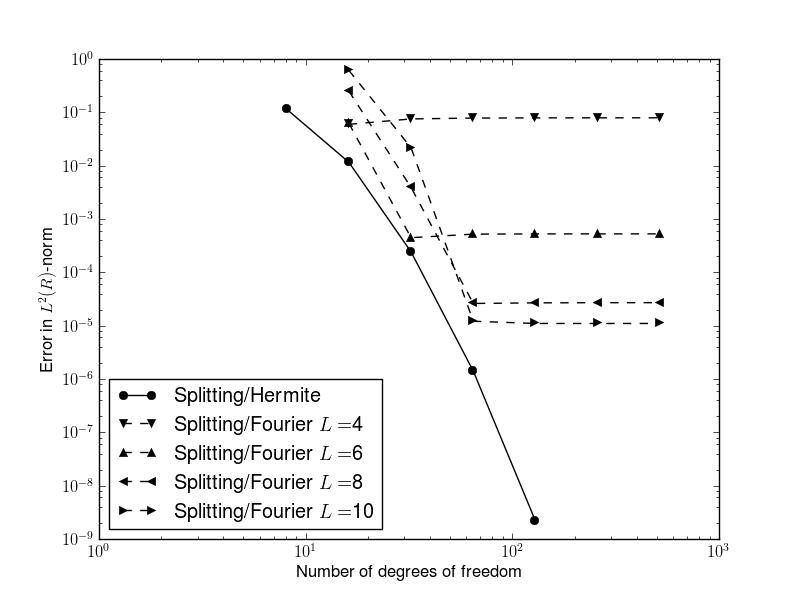}
  \captionof{figure}{Pathwise speed of convergence toward the exact solution}
  \label{fig:convK2}
\end{minipage}
\end{figure}
We can observe first that the discretization by finite differences is less efficient
than the spectral discretization for the same number of degrees of freedom.
We can also observe that for the two schemes using the spectral-Hermite discretization
behave similarly.
This is expected since they essentially use the same space discretization with
different time approximations, and the time step is chosen small enough to neglect
the errors coming from the time approximations.
We can also observe that the Fourier discretization converges around the same speed
of the Hermite one.
In some case (especially for small $T$), we observed that the convergence can be
even quicker.
We guess that this can be explained by the fact that the $n^{th}$ eigenvalue of
the Laplacian (on a bounded domain) is of order $n^2$, whereas it is of order $n$
for the harmonic oscillator.
Nevertheless, the Fourier discretization converges to a biased solution because
of the truncation of the level.

To take into account the fact that the Fourier solution is biased,
we compute the error between the solutions of the Fourier scheme with the most
precise solution computed with the Hermite discretization.
The result is plotted in Figure \ref{fig:convK2} for several domains $[-L,L]$
(for the Fourier discretization).

We observe now that, provided that the domain of integration is large enough,
the Fourier discretization performs quite well.
This remark is all the more interesting since the discrete Fourier transform
can be implemented very efficiently.
Nevertheless, since the spread of the solution is varying for each realization of the solution of \eqref{eq:GP}, the choice of the domain should be adapted to each realization.
This is the reason why the Hermite transform can be very interesting
(despite its cost).
Moreover, as we show thereafter, the time discretization is actually the limiting
approximation.
For example, in Figure \ref{fig:convK2}, it is useless to plot the error for
larger $L$ since the error generated by the time discretization for both splitting
methods becomes greater than the error generated by the space discretization.

We look now at the convergence with respect to the time discretization in the
case of a spectral Hermite discretization.
We compare the Crank-Nicolson discretization with the above time-splitting
discretization.
We choose a geometrically decreasing sequence of time steps, and we plot
the average over 100 samples of the square of the error in $L^2(\R)$-norm between
the solutions computed with two successive time steps.
The parameters are the number of modes $K\in\{40,80,120,160,200\}$,
and the coefficient
that we use in front of the Stratonovich integral
(we call it $\alpha$ and choose $\alpha\in\{0.2,0.4,0.6,1\}$).
We plot the results in Figure \ref{fig:convh} for $T=4$.
\begin{figure}[!ht]
    \centering
    \includegraphics[width=\textwidth]{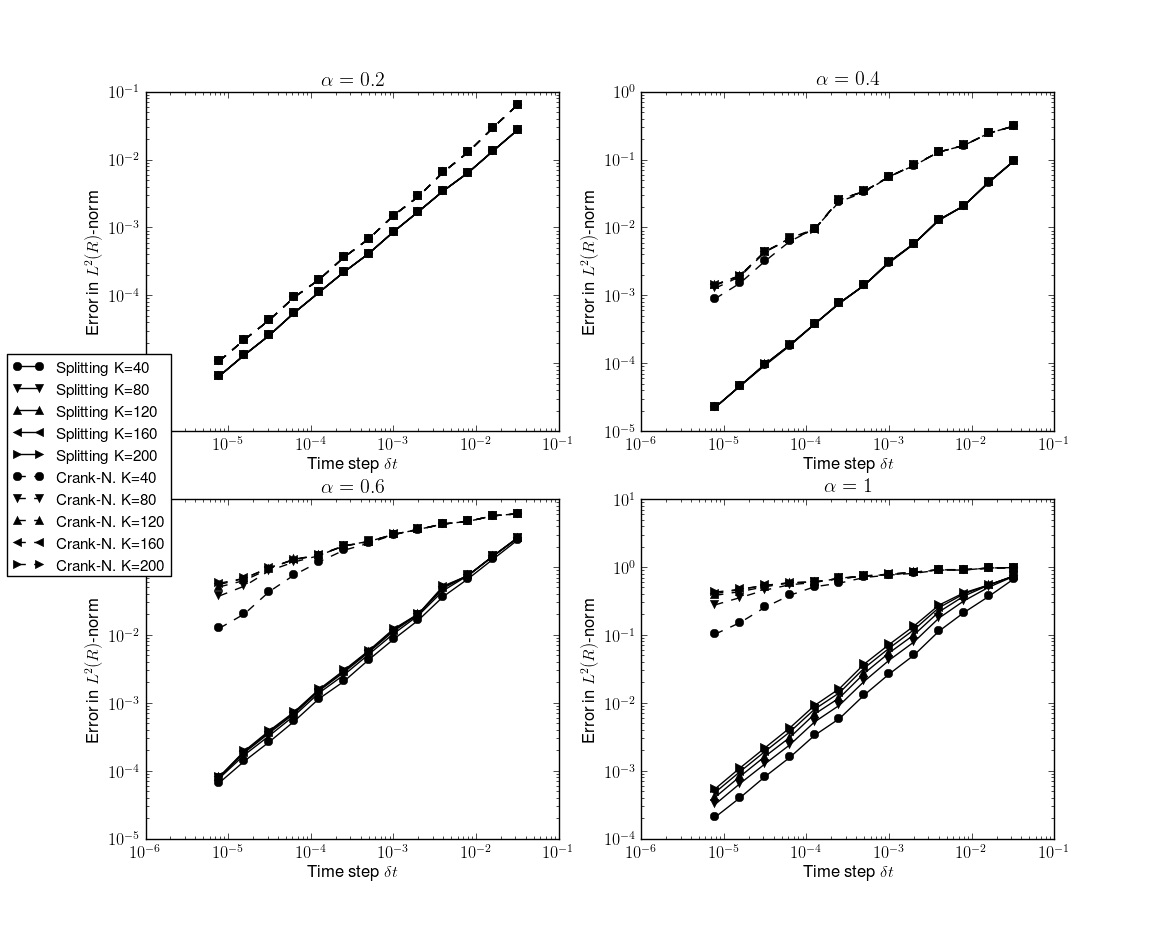}
    \caption{Mean square convergence with respect to the time step
    for the Crank-Nicolson and the splitting discretizations}
    \label{fig:convh}
\end{figure}
Each subplot corresponds to a value of $\alpha$, and in each subplot we plot
the mean square error for all the values of $K$, and for the two kinds of
discretization.
First, we can observe that the Splitting method is consistently doing better than
the Crank-Nicolson discretization.
Moreover, the latter is performing badly for high values of $\alpha$,
contrary to the former which is much less sensible with respect to this parameter.
In addition, the error is increasing with the number of modes for the large $\alpha$.
\begin{figure}[!ht]
    \centering
    \includegraphics[width=\textwidth]{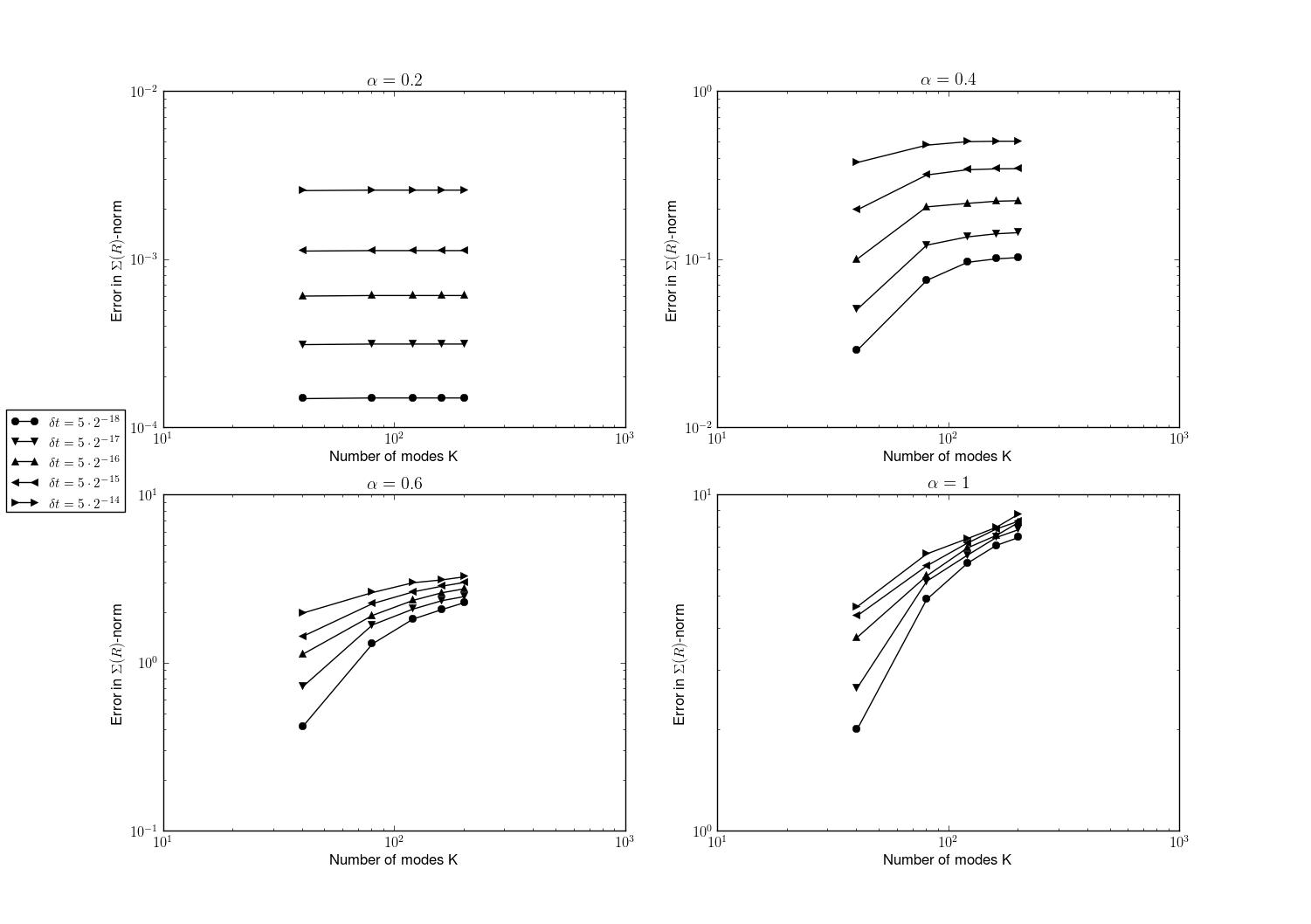}
    \caption{Mean square difference in $\Sigma$ norm between
    the solution of the Crank-Nicolson and the splitting discretizations
    with respect to the number of modes, for various time steps at time $T=5$.}
    \label{fig:stabK}
\end{figure}
In order to highlight this phenomena, we plot in Figure \ref{fig:stabK} the
mean-square difference in $\Sigma$ norm between
the solution of the Crank-Nicolson and the splitting discretizations
with respect to the number of modes, for various time steps at time $T=5$.
We observe that this difference is increasing with the number of modes
for a given time step, and decreasing with respect to the time step for a given
number of modes.
We recover here the sufficient condition used in the theoretical analysis that
imposes to the number of modes to increase slowly enough when the time step
decreases.
We believe that this numerical behaviour comes from the fact that the matrices $C(\delta t,\chi^{n+1})$
and $D(\delta t,\chi^{n+1})$ respectively defined by \eqref{eq:defC} and \eqref{eq:defD}
have diagonal terms of order $1\pm i\sqrt{\delta t}K/2$.
For $K=200$ and $\delta t=5\cdot 2^{-18}$ (which are the values used in Figure
\ref{fig:stabK}), we obtain that $\sqrt{\delta t}K/2\approx 0.44$, which is
not much smaller that one.
Thus it is possible that the Crank-Nicolson approximation is not precise enough
in this regime for the highest modes (that is to say that $K$ is too large
with respect to $\delta t$).
We recall that the theoretical analysis imposes that $K\delta t^{1/4}$ stays bounded.
Practically, this assumption imposes that the matrices $C(\delta t,\chi^{n+1})$
and $D(\delta t,\chi^{n+1})$ tend to be diagonal when the time step vanishes.

\vspace{2mm}
\paragraph{\textbf{Acknowledgement :}}
This work was supported by a public grant as part of the Investissement d'avenir project,
reference ANR-11-LABX-0056-LMH, LabEx LMH.
The author is grateful to A. de Bouard for helpful discussions.

\section{Appendix}
\label{sec:appendix}

\begin{proof}[\textbf{Proof of Proposition \ref{prop:exist_Sk}}]
  We begin with showing local existence and uniqueness in $\Sigma^k(\R^d)$, $\forall k\in\N^*$, $k\geq 2$, assuming that $\phi_0\in\Sigma^k(\R^d)$. Let $T_0>0$.
We use equation \eqref{eq:GP} with the following gauge transformation $\phi(t,x)=e^{-iG(t,x)}\psi(t,x)$ with $G(t,x)=\frac{1}{2}\abs{x}^2(t+\epsilon W(t))$.
The equation becomes
  \begin{align}
    \label{eq:GPgaugenl}
    i\partial_t\psi=-(\nabla-ix(t+\epsilon W(t)))^2\psi+\lambda \abs{\psi}^2\psi.
  \end{align}
  We denote by $U^\omega(t,s)$ the propagator for the linear equation,
  \begin{align}
    \label{eq:GPgaugel}
    i\partial_t\psi=-(\nabla-ix(t+\epsilon W(t)))^2\psi.
  \end{align}
  We introduce $T_\omega$, a positive random constant introduced in Proposition 4 \cite{dbf12},
  such that Propositions 6, 7 and Lemma 4.1 in \cite{dbf12} holds.
  In order to show the local existence and uniqueness in $\Sigma^k(\R^d)$,
  we use a classical fixed point argument based on application $\mathcal{T}^\omega$ defined by:
\begin{align*}
  (\mathcal{T}^\omega \psi)(t)=U^\omega(t,0)\phi_0-i\lambda\displaystyle\int_0^t U^\omega(t,s)\abs{\psi}^2\psi(s)ds.
\end{align*}
 To define the domains, we suppose that $(r,4)$ is an admissible pair,
 \textit{i.e.} $r=8/d$, and we set $\theta=d/4$.
 We define the following spaces for $T\leq T_0\wedge T_\omega$ where $T_0$ is fixed.
\begin{align*}
  X_T&=L^\infty(I,L^2)\cap L^r(I,L^4),\\
  Y^k_T&=\{v\in X_T /\quad x^\alpha\partial^\beta v\in X_T, \abs{\alpha}+\abs{\beta}\leq k\},\\
  \widetilde{Y}^k_T&=\{v /\quad x^\alpha\partial^\beta v\in L^1(I,L^2)+ L^{r'}(I,L^{\frac{4}{3}}), \abs{\alpha}+\abs{\beta}\leq k\},
\end{align*}
where $I=[0,T]$, and where $\alpha$ and $\beta$ are two multi-indices.
We now show that there exist $C_1(T_0,k,\omega)>0$, $C_2(T_0,k,\omega)>0$ and $C_3(T_0,k,\omega)>0$ so that for all $\psi_1,\psi_2\in Y^k_T$:
\begin{align}
  \norm{Y^k_T}{\mathcal{T}^\omega \psi_1}&\leq C_1 \normSigma{k}{\phi_0}+C_2T^{1-\theta}\norm{Y^k_T}{\psi_1}^3\label{eq:Tstable}\\
  \norm{L^r(I;L^4)}{\mathcal{T}^\omega \psi_1-\mathcal{T}^\omega \psi_2}&\leq C_3T^{1-\theta}(\norm{Y^k_T}{\psi_1}^2+\norm{Y^k_T}{\psi_2}^2)\norm{L^r(I;L^4)}{\psi_1-\psi_2}\label{eq:Tcontraction},
\end{align}
To prove this, we use the following lemmas:
\begin{lem}
  \label{lem:bound_NL}
  Let $k\in\N^*$, and assume $d\leq 3$.
  Then there exists $C(k)>0$ such that for all $\phi\in\Sigma^{k,4}(\R^d)\cap\Sigma^{k,2}(\R^d)$,
  \begin{align}
    \label{eq:bound_NL1}
    \normSigma{k,\frac{4}{3}}{\abs{\phi}^2\phi}&\leq C(k)\normSigma{k,4}{\phi}\normSigma{k,2}{\phi}^2,
  \end{align}
  and for all $\phi\in\Sigma^{k+1}(\R^d)$,
  \begin{align}
    \label{eq:bound_NL2}
    \normSigma{k+1}{\abs{\phi}^2\phi}&\leq C(k)\normSigma{k,4}{\phi}^2\normSigma{k+1}{\phi}.
  \end{align}
\end{lem}
\begin{lem}
  \label{lem:ULk}
  Let $k\in\N^*$. Then, for all $\phi_0\in\Sigma^k(\R^d)$, $t\mapsto U^\omega(t,0)\phi_0\in\mathcal{C}(I,\Sigma^k(\R^d))\cap Y^k_T$. Moreover,
  \begin{align*}
    \norm{Y^k_T}{U^\omega(\cdot,0)\phi_0}\leq C_1(T_0,k,\omega)\normSigma{k}{\phi_0}.
  \end{align*}
  For all $\phi\in Y^k_T$, $\abs{\phi}^2\phi\in \widetilde{Y}^k_T$ and
  \begin{align*}
    \norm{Y^k_T}{\Lambda_\omega(\cdot,0)\abs{\phi}^2\phi}\leq C_2(T_0,k,\omega)\norm{\widetilde{Y}^k_T}{\abs{\phi}^2\phi},
  \end{align*}
  where $\Lambda_\omega(\cdot,0) \abs{\phi}^2 \phi $ is defined for all $t\in I$ by,
  \begin{align*}
    \Lambda_\omega(t,0)=\int_0^t U^\omega(t,s) \abs{\phi_s}^2\phi_s ds.
  \end{align*}
\end{lem}
\begin{proof}[Proof of lemma \ref{lem:bound_NL}]
  Equation \eqref{eq:bound_NL2} can be shown using Sobolev embedding theorems and Holder's
  inequality.
  To show equation \eqref{eq:bound_NL1}, we use the fact that there exists $C(k)>0$
  such that for all multi-indices $\alpha$ and $\beta$ such that $\abs{\alpha}+\abs{\beta}\leq k$,
  \begin{align*}
    \abs{x^\alpha\partial^\beta(\abs{\phi}^2\phi)}\leq C(k)\displaystyle\sum_{\abs{l}\leq \abs{\beta},\abs{j}\leq k-1}\abs{x^\alpha\partial^l\phi}\abs{\partial^j\phi}^2,
  \end{align*}
  which leads, using H\"{o}lder's inequality and the Sobolev embedding $\Sigma^{k}(\R^d)\hookrightarrow H^k(\R^d)\hookrightarrow W^{k-1,4}(\R^d)$, to
  \begin{align*}
    \normL{4/3}{x^\alpha\partial^\beta(\abs{\phi}^2\phi)}\leq C(k)\norm{\Sigma^{k,4}}{\phi}\norm{\Sigma^{k,2}}{\phi}^2.
  \end{align*}
\end{proof}
\begin{proof}[Proof of lemma \ref{lem:ULk}]
  By using $k$ times the Proposition 6 of \cite{dbf12}, we get that
  for all multi-index $\alpha,\beta$ such that $\abs{\alpha}+\abs{\beta}\leq k$,
  \begin{align*}
    \norm{X_T}{x^\alpha\partial^\beta U^\omega(\cdot,0)\phi_0}\leq C(k,T_0)\displaystyle\sum_{\abs{\gamma}+\abs{\delta}\leq k}\norm{X_T}{U^\omega(\cdot,0)x^\gamma\partial^\delta \phi_0},
  \end{align*}
  and using Proposition 7 \cite{dbf12} that relies on the fact that $U^\omega$
  in an isometry in $L^2(\R^d)$ it comes,
  \begin{align*}
    \norm{X_T}{x^\alpha\partial^\beta U^\omega(\cdot,0)\phi_0}\leq C(k,T_0)\displaystyle\sum_{\abs{\gamma}+\abs{\delta}\leq k}\normL{2}{x^\gamma\partial^\delta \phi_0},
  \end{align*}
which proves the first point. The second one can be shown in a similar way.
\end{proof}
To show \eqref{eq:Tstable}, Lemma \ref{lem:bound_NL} and H\"{o}lder's inequality give
\begin{align*}
  \norm{\widetilde{Y}^k_T}{\abs{\phi}^2\phi}&\leq C(k)T^{1-\theta}\norm{L^{r}(I;\Sigma^{k,4})}{\phi}\norm{L^\infty(I;\Sigma^{k,2})}{\phi}^2\\
  \norm{\widetilde{Y}^k_T}{\abs{\phi}^2\phi}&\leq C(k)T^{1-\theta}\norm{Y^k_T}{\phi}^3
\end{align*}
To show \eqref{eq:Tcontraction}, we use the fact that there exists $C>0$ so that for all $u,v\in\C$,
\begin{align*}
  \abs{\abs{u}^2u-\abs{v}^2v}\leq C (\abs{u}^2+\abs{v}^2)\abs{u-v}
\end{align*}
 Using Strichartz estimates given in Proposition 7 \cite{dbf12} and H\"{o}lder inequality,
 \begin{align*}
   \norm{L^r(I;L^4)}{\mathcal{T}^\omega\psi-\mathcal{T}^\omega\phi}&\leq C(\omega,T_0)\norm{L^{r'}(I;L^{4/3})}{\abs{\psi}^2\psi-\abs{\phi}^2\phi}\\
   &\leq C(\omega,T_0)(\norm{L^{\infty}(I;L^{4})}{\psi}^2+\norm{L^{\infty}(I;L^{4})}{\phi}^2)\norm{L^{r'}(I;L^4)}{\psi-\phi}\\
   &\leq T^{1-\theta}C(\omega,T_0)(\norm{Y^k_T}{\psi}^2+\norm{Y^k_T}{\phi}^2)\norm{L^{r}(I;L^4)}{\psi-\phi}
 \end{align*}
Then $T$ can be chosen small enough so that $\mathcal{T}^\omega$ is a contraction in
\begin{align*}
  B_M:=\left\{\phi\in Y^k_T /\quad\norm{Y^k_T}{\phi}\leq M \right\}
\end{align*}
with $M=2C(T_0,k,\omega)\normSigma{k}{\phi_0}$.

We can show now by induction that this local solution is actually global in time.
Let $m\in\N$ such that $2\leq m<k$.
Since we know that the solution is global in the case $m=2$
(see Proposition 4 \cite{Bouard20152793}), we are going to show that
if the solution is global in $\Sigma^m$, then it is global in $\Sigma^{m+1}$.
To prove it, we begin with showing that it is global in $\Sigma^{m,4}(\R^d)$.
It can be shown by induction using proposition 6 \cite{dbf12} for $\phi\in Y^k_T$
and all multi-index $\alpha,\beta$ so that $\abs{\alpha}+\abs{\beta}\leq m$ and $\forall (s,t)\in [0,T]^2$, $s<t$:
\begin{align*}
  \normL{4}{x^\alpha\partial^\beta U^\omega(t,s)\phi(s)}\leq C(\omega,T_0,k)\left((t-s)^m\vee 1 \right)\displaystyle\sum_{\abs{\delta}+\abs{\gamma}\leq m}\normL{4}{U^\omega(t,s)x^\delta\partial^\gamma \phi(s)}
\end{align*}
Then, using lemma 4.1 \cite{dbf12},
\begin{equation}
  \label{eq:xdu}
  \begin{aligned}
    \normL{4}{x^\alpha\partial^\beta\displaystyle\int_0^t U^\omega(t,s)\abs{\phi(s)}^2\phi(s)ds}&\leq C(\omega,T_0,k)\displaystyle\int_0^t\left((t-s)^m\vee 1 \right)
    \displaystyle\sum_{\abs{\delta}+\abs{\gamma}\leq m}\abs{t-s}^{-d/4}\normL{4/3}{x^\delta\partial^\gamma \abs{\phi(s)}^2\phi(s)}ds,\\
    \normL{4}{x^\alpha\partial^\beta\displaystyle\int_0^t U^\omega(t,s)\abs{\phi(s)}^2\phi(s)ds}&\leq C(\omega,T_0,k)\displaystyle\int_0^t\left((t-s)^m\vee 1 \right)\abs{t-s}^{-d/4}\norm{\Sigma^{m,4/3}}{\abs{\phi(s)}^2\phi(s)}ds.
  \end{aligned}
\end{equation}
We notice that $\abs{t-s}^{-d/4}$ is integrable in $s$ since $d\leq 3$.
Lemma \ref{lem:bound_NL} and Gronwall's inequality enables us to conclude that the local solution is
actually global in $\Sigma^{m,4}(\R^d)$.

We show now that the local solution is actually global in $\Sigma^{m+1}(\R^d)$ using again the mild formulation:
\begin{align*}
  \normSigma{m+1}{\psi(t)}\leq \normSigma{m+1}{U^\omega(t,0)\phi_0}+\normSigma{m+1}{\Lambda^\omega(t,0)(\abs{\psi}^2\psi)}.
\end{align*}
Using proposition 7 \cite{dbf12}, and reasoning in a similar way than previously, we can show:
\begin{align*}
  \normSigma{m+1}{\psi(t)}&\leq C(\omega,T_0,k)\normSigma{m+1}{\phi_0}+C(\omega,T_0,k)\displaystyle\int_0^t \normSigma{m+1}{\abs{\psi(s)}^2\psi(s)}ds.
\end{align*}
We now use Lemma \ref{lem:bound_NL} and Gronwall's inequality to conclude that
the solution is global in $\Sigma^{m+1}$.
\end{proof}
\bibliographystyle{amsplain}
\bibliography{bibliographie}

\end{document}